\documentclass[reqno,11pt]{amsart} \numberwithin{equation}{section}
\input{amssym.def}
\input{amssym.tex}
\usepackage{graphicx}
\begin{document}
\title[Extragradient and linesearch and algoritms for solving... ]
{Extragradient and linesearch algorithms for solving equilibrium problems and fixed point problems in Banach spaces}
\author[Z. Jouymandi and F. Moradlou]{Zeynab Jouymandi$^1$ and Fridoun Moradlou$^2$}
\address{\indent $^{1,2}$Department of Mathematics
\newline \indent Sahand University of Technology
\newline \indent Tabriz, Iran}
\email{\rm $^1$ z\_jouymandi@sut.ac.ir \& z.juymandi@gmail.com}
\email{\rm $^2$ moradlou@sut.ac.ir \& fridoun.moradlou@gmail.com}
\thanks{}
\begin{abstract}
In this paper, using sunny generalized nonexpansive retraction, we
propose new extragradient and linesearch algorithms for finding a
common element of the set of solutions of an equilibrium problem
and the set of fixed points of a relatively nonexpansive mapping
in Banach spaces. To prove strong convergence of iterates in the
extragradient method, we introduce a $\phi$-Lipschitz-type
condition and assume that the equilibrium bifunction satisfies in
this condition. This condition is unnecessary when the linesearch
method is used instead of the extragradient method. A numerical
example is given to illustrate the usability of our results. Our
results generalize, extend and enrich some existing results in the
literature.
\end{abstract}


\maketitle

\baselineskip=15.8pt \theoremstyle{definition}
  \newtheorem{df}{Definition}[section]
    \newtheorem{rk}{Remark}[section]
\theoremstyle{plain}
 \newtheorem{lm}{Lemma}[section]
  \newtheorem{thm}{Theorem}[section]
  \newtheorem{exa}{Example}[section]
  \newtheorem{cor}{Corollary}[section]
\newtheorem{prop}{Proposition}[section]
\setcounter{section}{0}
%\renewcommand{\baselinestretch}{3}
%%%%%%%%%%%%%%%%%%%%%%%%%%%%%%%%%%%%%%%%%%%%%%%%%%%%%%
\section{Introduction}
%\addtolength{\baselineskip}{2ex}
In this paper, we consider the following equilibrium problem ($EP$)
in the sense of Blum and Otteli \cite{Blum}, which consists in
finding a point $x^{*}\in C$ such that
\begin{equation*}\label{vthoe1}
f(x^{*},y)\geq 0, \: \:\ \forall \ y\in C,
\end{equation*}
where $C$ is a nonempty, closed and convex subset of a real Banach
space $E$ and \linebreak
$f:C\times C\rightarrow \mathbb{R}$ is an equilibrium
bifunction, i.e., $f(x,x)=0$ for all $x\in C$. The solution set of
($EP$) is denoted by $E(f)$. The equilibrium problem which also
known under the name of Ky Fan inequality \cite{Fan} covers, as special
cases, many well-known problems, such as the optimization problem,
the variational inequality problem and nonlinear complementarity
problem, the saddle point problem, the generalized Nash equilibrium
problem in game theory, the fixed point problem and others; (see
\cite{Mu,Stro}). Also numerous problems in physic and
economic reduce to find a solution of an equilibrium problem. Many
methods have been proposed to solve the equilibrium problems see for
instance \cite{Blum,Mou,Mu,Van,Vuo}. In 1980,\linebreak Cohen \cite{Coh}
introduced a useful tool for solving optimization problem which is
known as auxiliary problem principle and extended it to
variational inequality \cite{Cohe}. In auxiliary problem principle a
sequence $\{x_k\}$ is generated as follows: $x_{k+1} \in C$ is a
unique solution of the following strongly convex problem
\begin{equation}\label{equ.1}
\min\Big\{c_k f(x_k, y) + \frac{1}{2} \|x_k - y\|\Big\},
\end{equation}
where $c_k >0$. Recently, Mastroeni \cite{Mas} extended the auxiliary
problem principle to equilibrium problems under the assumptions that
the equilibrium function $f$ is strongly monotone on $C\times C$ and
that $f$ satisfies the following Lipschitz- type condition:
\begin{equation}\label{equ.2}
f(x,y)+f(y,z)\geq f(x,z)-c_{1}\|y-x\|^{2}-c_{2}\|z-y\|^{2},
\end{equation}
for all $x,y,z \in C$ where $c_1, c_2>0$. To avoid the monotonicity
of $f$, motivated by Antipin \cite{Ant}, Tran et al. \cite{Tr} have used
an extrapolation step in each iteration after solving (\ref{equ.1})
and suppose that $f$ is pseudomonotone on $C\times C$ which is
weaker than monotonicity assumption. They assumed $y_k$ was the
unique solution of (\ref{equ.1}) and the unique solution of the
following strongly convex problem
\begin{equation*}
\min\Big\{c_k f(y_k, y) + \frac{1}{2} \|y - x_k\|\Big\},
\end{equation*}
is denoted by $\{x_{k+1}\}$. In special case , when the problem
$(EP)$ is a variational inequality problem, this method reduces to
the classical extragradient method which has been introduced by
Korpelevich \cite{korp}.
%%%%%%%%%%%%%%%%%%%%%%%%%%%%%%%%%%%%%%%%%%%%%%%%%%%%%%%%%%%%%%%%%%%%%
The extragradient method is well known because of its efficiency
in numerical tests. In the recent years, many authors obtained
extragradient algorithms for solving ($EP$) in Hilbert spaces
where convergence of the proposed algorithms was required $f$ to
satisfy a certain Lipschitz-type condition \cite{Ng,Tr,Vuo}.
Lipschitz-type condition depends on two positive parameters
$c_{1}$ and $c_{2}$ which in some cases, they are unknown or
difficult to approximate. In other to avoid this requirement,
authors used the linesearch technique in a Hilbert space to obtain
convergent algorithms for solving equilibrium problem
\cite{Ng,Tr,Vuo}.
%%%%%%%%%%%%%%%%%%%%%%%%%%%%%%%%%%%%%%%%%%%%%%%%%%%%%%%%%%%%%%%%%%%%
\par
In this paper, we consider the following auxiliary equilibrium
problem ($AUEP$) for finding $x^{*}\in C$ such that
\begin{equation}\label{vthoe2}
\rho f(x^{*},y)+L(x^{*},y)\geq 0,
\end{equation}
for all $y\in C$, where $\rho>0$ is a regularization parameter and
$L:C\times C\rightarrow \mathbb{R}$ be a nonnegative differentiable
convex bifunction on $C$ with respect to the second argument $y$,
for each fixed $x\in C$, such that
 \begin{enumerate}
\item[(i)]$L(x,x)=0$ for all $x\in C$,
\item[(ii)]$\nabla_{2}L(x,x)=0$ for all $x\in C$.
\end{enumerate}
Where $\nabla_{2}L(x,x)$ denotes the gradient of the function
$L(x,.)$ at $x$.
\par
In the recent years, many authors studied the problem of finding a
common element of the set of fixed points of a nonlinear mapping
and the set of solutions of an equilibrium problem in the
framework of Hilbert spaces and Banach spaces, see for instance
\cite{Ceng,Qin,tak,Ta,Vuo}. In all of these methods, authors have
used metric projection in Hilbert spaces and generalized metric
projection in Banach spaces.
\par
In this paper, motivated D. Q. Tran et al. \cite{Tr} and P. T.
vuong et al. \cite{Vuo}, we introduce new extragradient and
linesearch algorithms for finding a common element of the set of
solutions of an equilibrium problem and the set of fixed points of
a relatively nonexpansive mapping in Banach spaces, by using sunny
generalized nonexpansive retraction. Using this method, we prove
strong convergence theorems under suitable conditions.
%%%%%%%%%%%%%%%%%%%%%%%%%%%%%%%%%%%%%%%%%%%%%%%%%%%%%%%%%%%%%%%%
\section{Preliminaries}
 We denote by $J$ the normalized duality mapping from $E$ to $2^{E^{*}}$ defined by
\[Jx=\lbrace x^{*}\in E^{*} :\langle x , x^{*} \rangle=\Vert x\Vert^{2}=\Vert x^{*}\Vert^{2}\rbrace, \  \forall  \ x\in E.\]
Also, denote the strong convergence and the weak convergence of a sequence $\{x^{k}\}$ to $x$ in $E$ by $x^{k}\rightarrow x$ and $x^{k}\rightharpoonup x$, respectively.\\
 %denote the weak$^{^{*}}$ convergence of a sequence $\{x^{*^{k}}\}$ to $x^{*}$ in $E^{*}$ by $x^{*^{k}}\rightharpoonup^{*} x^{*}$ and use the notation $\|.\|$ for norm.\\
Let $S(E)$ be the unite sphere centered at the origin of $E$. A Banach space $E$ is strictly convex if $\|\frac{x+y}{2}\|<1$, whenever $x, y \in S(E)$ and $x\neq y$.
Modulus of convexity of $E$ is defined by $$\delta_{E}(\epsilon)=inf \{1-\frac{1}{2}\|(x+y)\|:\ \|x\|, \|y\|\leq 1, \ \|x-y\|\geq \epsilon\}$$ for all $\epsilon\in [0, 2]$. $E$ is said to be uniformly convex if $\delta_{E}(0)=0$ and $\delta_{E}(\epsilon)>0$ for all $0<\epsilon\leq2$. Let $p$ be a fixed real number with $p\geq2$. A Banach space $E$ is said to be $p$-uniformly convex \cite{kato} if there exists a constant $c>0$ such that $\delta_{E}\geq c\epsilon^{p}$ for all $\epsilon\in [0, 2]$.
The Banach space $E$ is called smooth if the limit
 \begin{equation}\label{eq1}
 \lim_{t\rightarrow0}\frac{\|x+ty\|-\|x\|}{t},
 \end{equation}
 exists for all $x,y\in S(E)$. It is also said to be uniformly smooth if the limit
  $(\ref{eq1})$ is attained uniformly for all $x,y\in S(E)$. Every uniformly smooth Banach space $E$ is smooth.
If a Banach space $E$ uniformly convex, then $E$ is reflexive and strictly convex \cite{Agarwal,Takahashi}.
  \par
  Many properties of the normalized duality mapping $J$ have been given in \cite{Agarwal,Takahashi}.\\ We give some of those in the following:
  \begin{enumerate}
%\item $J(0) =\{0\}$.
\item For every $x\in E$, $Jx$ is nonempty closed convex and bounded subset of $E^*$.
\item If $E$ is smooth or $E^*$ is strictly convex, then $J$ is single-valued.
\item If $E$ is strictly convex, then $J$ is one-one.
%i.e., if $x\neq y$ then $ Jx\cap Jy=\phi$.
\item If $E$ is reflexive, then $J$ is onto.
%\item If $E$ is smooth, then $J$ is single-valued.
\item If $E$ is strictly convex, then $J$ is strictly monotone, that is,
$$\langle x-y, Jx-Jy\rangle>0,$$
for all  $x, y\in E$ such that $x\neq y$.
%\item if $E$ is smooth and reflexive, then $J$ is norm-to-weak$^{*}$ continuous, that is, $Jx^{k}\rightharpoonup^{*}Jx$ whenever $x^{k}\rightarrow x$.
\item If $E$ is smooth, strictly convex  and reflexive  and $J^*:E^*\rightarrow 2^E$ is the normalized duality mapping on $E^*$, then $J^{-1}=J^*$, $JJ^*=I_{E^*}$ and $J^*J=I_{E}$, where $I_{E}$ and $I_{E^*}$ are the identity mapping on $E$ and $E^*$, respectively.
\item If $E$ is  uniformly convex and uniformly smooth, then $J$ is uniformly norm-to-norm continuous on bounded sets of $E$ and $J^{-1}=J^*$ is also uniformly norm-to-norm continuous on bounded sets of $E^*$, i.e., for $\varepsilon>0$ and $M>0$, there is a $\delta>0$ such that
\begin{equation}\label{eq30}
\|x\|\leq M, \ \|y\|\leq M \ \ \text{and} \ \ \|x-y\|<\delta \ \ \Rightarrow \ \ \|Jx-Jy\|<\varepsilon,
\end{equation}
\begin{equation}\label{eq18}
\|x^{*}\|\leq M, \ \|y^{*}\|\leq M \ \ \text{and} \ \ \|x^{*}-y^{*}\|<\delta \ \ \Rightarrow \ \ \|J^{-1}x^{*}-J^{-1}y^{*}\|<\varepsilon.
\end{equation}
\end{enumerate}
\par
Let $E$ be a smooth Banach space, we define the function $\phi: E\times E\rightarrow \mathbb{R}$ by \[\phi(x, y)=\Vert x\Vert^{2}-2\langle x,Jy \rangle+\Vert y\Vert^{2},\]
for all $x, y\in E$. Observe that, in a Hilbert space $H$, $\phi(x,y)=\|x-y\|^{2}$ for all $x,y\in H.$\\
It is clear from definition of $\phi$ that for all $x, y, z, w\in E$,
\begin{equation}\label{eq14}
(\Vert x\Vert -\Vert y\Vert)^{2} \leq \phi(x,y)\leq (\Vert x\Vert+\Vert y\Vert)^{2},
\end{equation}
\begin{equation}\label{eq36}
\phi(x,y)=\phi(x,z)+\phi(z,y)+2\langle x-z,Jz-Jy\rangle,
\end{equation}
\begin{equation}\label{equa6}
2\langle x-y,Jz-Jw\rangle=\phi(x,w)+\phi(y,z)-\phi(x,z)-\phi(y,w).
\end{equation}
If $E$ additionally assumed to be strictly convex, then
\begin{equation}\label{equa21}
\phi(x,y)=0\:\:\:\Longleftrightarrow\:\:\: x=y.
\end{equation}
Also, we define the function $V:E\times E^{*}\rightarrow \mathbb{R}$ by $V(x, x^{*})=\| x\|^{2}-2<x, x^{*}>+\| x^{*} \|^{2}$,\\ for all $x\in E$ and $x^{*}\in E$. That is, $V(x,x^{*})=\phi(x,J^{-1}x^{*})$ for all $x\in E$ and $x\in E^{*}$.\\
It is well known that, if $E$ is a reflexive strictly convex and smooth Banach space with $E^{*}$ as its dual, then
\begin{equation}\label{eq5}
 V(x, x^{*})+2\langle J^{-1}x^{*}-x,y^{*}\rangle \leq V(x, x^{*}+y^{*}),
 \end{equation}
 for all $x\in E$ and all $x^{*},y^{*}\in E^{*}$ \cite{Rockfellar}.
\par
Let $E$ be a smooth Banach space and $C$ be a nonempty subset of $E$. A mapping $T: C\rightarrow C$ is called generalized nonexpansive \cite{Ibaraki} if $F(T)\neq \emptyset$ and $$\phi(y,Tx)\leq \phi(y,x),$$
for all $x\in C$ and all $y\in F(T).$\\
Let $C$ be a closed convex subset of $E$ and $T:C\rightarrow C$ be a mapping. A point $p$ in $C$ is said to be an asymptotic fixed point of $T$ if $C$ contains a sequence $\{x^{k}\}$ which converges weakly to $p$ such that $\lim\limits_{k\rightarrow\infty}(Tx^{k}-x^{k})=0$. The set of asymptotic fixed points of $T$ will be denoted by $\hat{F}(T)$.
A mapping $T:C\rightarrow C$ is called relatively nonexpansive if $\hat{F}(T)=F(T)$ and $\phi(p,Tx)\leq\phi(p,x)$ for all $x\in C$ and $p\in F(T)$. The asymptotic behavior of relatively nonexpansive mappings was studied in \cite{But}. $T$ is said to be relatively quasi-nonexpansive if $F(T)\neq\emptyset$ and $\phi(p,Tx)\leq\phi(p,x)$ for all $x\in C$ and all $p\in F(T)$. The class of relatively quasi-nonexpansive mapping is broader than the class of relatively nonexpansive mappings which requires $\hat{F}(T)=F(T)$.\\
It is well known that, if $E$ is a strictly convex and smooth Banach space, $C$ is a nonempty closed convex subset of $E$ and $T:C\rightarrow C$ is a relatively quasi-nonexpansive mapping, then $F(T)$ is a closed convex subset of $C$ \cite{Q}.
\par
Let $D$ be a nonempty subset of a Banach space $E$. A mapping $R:E\rightarrow D$ is said to be sunny \cite{Ibaraki} if $$R(Rx+t(x-Rx))=Rx,$$
for all $x\in E$ and all $t\geq 0$. A mapping $R:E\rightarrow D$ is said to be a retraction if $Rx=x$ for all $x\in D$.
$R$ is a sunny nonexpansive retraction from $E$ onto $D$ if $R$ is a retraction which is also sunny and nonexpansive.
A nonempty subset $D$ of a smooth Banach space $E$ is said to be a generalized nonexpansive retract (resp. sunny generalized nonexpansive retract) of $E$ if there exists a generalized nonexpansive retraction (resp. sunny generalized nonexpansive retraction) $R$ from $E$ onto $D$.\\
If $E$ is a smooth, strictly convex and reflexive Banach space,
$C^{*}$ be a nonempty closed convex subset of $E^{*}$ and
$\Pi_{C^{*}}$ be the generalized metric projection of $E^{*}$ onto
$C^{*}$. Then the $R=J^{-1}\Pi_{C^{*}}J$ is a sunny generalized
nonexpansive retraction of $E$ onto $J^{-1}C^{*}$ \cite{Kohsa}.
\begin{rk}
If $E$ is a Hilbert space. Then $R_{C}=\Pi_{C}=P_{C}$.
\end{rk}
We need the following lemmas for the proof of our main results.\\
If $C$ is a convex subset of Banach space $E$, then we denote by $N_{C}(\nu)$ the normal cone for $C$ at a point $\nu\in C$, that is \[N_{C}(\nu):=\{x^{*}\in E^{*}: \langle \nu-y,x^{*}\rangle\geq0,\ \forall  \ y\in C\}.\]
Suppose that $E$ is a Banach space and let $f:E\rightarrow (-\infty,+\infty]$ be a paper function. For $x_{0}\in D(f)$, we define the subdifferential of $f$ at $x_{0}$ as the subset of $E^{*}$ given by
$$\partial f(x_{0})=\{x^{*}\in E^{*}: \: f(x)\geq f(x_{0})+\langle x^{*},x-x_{0}\rangle, \: \forall x\in E\}.$$
If $\partial f(x_{0})\neq \emptyset$, then we say $f$ is subdifferentiable at $x_{0}$.
\begin{lm}\label{lem7}
Let $C$ be a nonempty convex subset of a Banach space $E$ and $f:E\rightarrow\mathbb{R}$ be a convex and subdifferentiable function, then $f$ is minimized at $x\in E$ if and only if $$0\in \partial f(x)+N_{C}(x).$$
\end{lm}
%\begin{proof}
%Clearly, it following from definitions $N_{C}(x)$ and $\partial f(x)$.
%\end{proof}
\begin{lm}\label{lem8}\cite{Au}
Let $E$ be a reflexive Banach space. If $f:E\rightarrow \mathbb{R}\cup\{+\infty\}$ and $g:E\rightarrow \mathbb{R}\cup\{+\infty\}$ are nontrivial, convex and lower continuous functions and if $\:0\in Int (Dom f-Dom g)\:$, then
$$\partial(f+g)(x)=\partial f(x)+\partial g(x).$$
\end{lm}
\begin{lm}\label{lem20}\cite{Au}
Suppose that a convex function $f$ is continuouse on the interior of its domain. Then, for all $x\in Int\: (Dom f)$, $\partial f(x)$
is non-empty and bounded.
\end{lm}
\begin{lm}\label{lem13}
Let $C$ be a nonempty convex subset of a Banach space $E$ and let $f:C\times C\rightarrow\mathbb{R}$ be an equilibrium bifunction and convex respect to the second variable. then the following statements are equivalent:
\begin{enumerate}
\item[(i)] $x^{*}$ is a solution to $E(f)$,
\item[(ii)] $x^{*}$ is a solution to the problem
$$\min_{y\in C}f(x^{*},y).$$
\end{enumerate}
\end{lm}
\begin{proof}
Using Lemma \ref{lem7}, we get desired results.
\end{proof}
Equivalence between $E(f)$ and ($AUEP$) is stated in the following lemma.
\begin{lm}\label{lem2}
Let $C$ nonempty, convex and closed subset of a reflexive Banach space $E$ and $f:C\times C\rightarrow\mathbb{R}$ be an equilibrium bifunction and let $x^{*}\in C$. suppose that $f(x^{*},.):C\rightarrow\mathbb{R}$ is convex and subdifferentiable on $C$. Let $L:C\times C\rightarrow\mathbb{R}_{+}$ be a differentiable convex function on $C$ with respect to the second argument $y$ such that
\begin{enumerate}
\item[(i)] $L(x^{*},x^{*})=0$,
\item[(ii)] $\nabla_2L(x^{*},x^{*})=0$.
\end{enumerate}
Then $x^{*}\in C$ is a solution to $E(f)$ if and only if $x^{*}$ is a solution to ($AUEP$).
\end{lm}
\begin{proof}
It is clear from lemmas \ref{lem8} and \ref{lem13}.
\end{proof}
%\begin{lm}\cite{Agarwal}
%Let $E$ be a topological space and $f:E\rightarrow (-\infty,\infty]$ be a function. then the following statements are equivalent:
%\begin{enumerate}
%\item $f$ is lower semicontinuous.
%\item For each $\alpha\in \mathbb{R}$, the level set $\{x\in E: \ \ f(x)\leq\alpha\}$ is closed.
%\item The epigraph of the function $f$, $\{(x,\alpha)\in E\times \mathbb{R}: \ \ f(x)\leq\alpha\}$ is closed.
%\end{enumerate}
%\end{lm}
%\begin{lm}\cite{Agarwal}\label{eq23}
%Let $C$ be nonempty closed convex subset of a Banach space $E$ and\linebreak $f:E\rightarrow (-\infty,\infty]$ be a convex function. Then $f$ is lower semicontinuous in the norm topology if and only if $f$ is lower semicontinuous in the weak topology.
%\end{lm}
\begin{lm}\cite{Ibaraki}
Let $C$ be a nonempty closed sunny generalized nonexpansive retract of a smooth and strictly convex Banach space $E$. Then the sunny generalized nonexpansive retraction from $E$ onto $C$ is uniquely determined.
\end{lm}
\begin{lm}\cite{Ibaraki}\label{lem4}
Let $C$ be a nonempty closed subset of a smooth and strictly convex Banach space $E$ such that there exists a sunny generalized nonexpansive retraction $R$ from $E$ onto $C$ and let $(x,z)\in E\times C$. Then the following hold:
\begin{enumerate}
\item $z=Rx $ if and only if $\langle x-z,Jy-Jz\rangle\leq 0$ for all $y\in C$,
\item $\phi(z,Rx)+\phi(Rx,x)\leq\phi(z,x).$
\end{enumerate}
\end{lm}
\begin{lm}\cite{Kohsa}\label{eq38}
Let $E$ be a smooth, strictly convex and reflexive Banach space and $C$ be a nonempty closed sunny generalized nonexpansive retract of $E$. Let $R$ be the sunny generalized nonexpansive retraction from $E$ onto $C$ and $(x,z)\in E\times C$. Then the following are equivalent:
\begin{enumerate}
\item $z=Rx$,
\item $\phi(z,x)=\min_{y\in C}\phi(y,x)$
\end{enumerate}
\begin{lm}\cite{Ibaraki}\label{Ibaraki}
Let $E$ be a smooth, strictly convex and reflexive Banach space and let $C$ be a nonempty closed subset of $E$. Then the following are equivalent:
\begin{enumerate}
\item $C$ is a sunny generalized nonexpansive retract of $E$,
\item $C$ is a generalized retract of $E$,
\item $JC$ is closed and convex.
\end{enumerate}
\end{lm}
\end{lm}
\begin{lm}\cite{Yao}\label{eq6}
Let $E$ be a $2$-uniformly convex and smooth Banach space. Then, for all $x, y\in E$, we have $$\|x-y\|\leq\frac{2}{c^{2}}\|Jx-Jy\|,$$
where $\frac{1}{c}(0\leq c\leq 1)$ is the $2$-uniformly convex constant of $E$.
\end{lm}
\begin{lm}\label{lem16}\cite{ch}
Suppose $p>1$ is a real number, then the following are equivalent
\begin{enumerate}
\item[(i)] $E$ is a $p$-uniformly convex Banach space,
\item[(ii)] there exits $\tau>0$, such that for each $f_{x}\in J_{p}(x)$ and $f_{y}\in J_{p}(y)$, we have
$$\langle x-y,f_{x}-f_{y}\rangle\geq\tau\|x-y\|^{p}.$$
\end{enumerate}
Where $J_{p}(x)=\{x^{*}\in E^{*}: \langle x,x^{*}\rangle=\|x\|\|x^{*}\|, \: \|x^{*}\|=\|x\|^{p-1}\}$.
\end{lm}
\begin{lm}\label{kamimra}\cite{Kamimura}
Let $E$ be a uniformly convex Banach space and let $r>0$. Then there exists a strictly increasing, continuous and convex function $g:[0, 2r]\rightarrow [0, \infty)$, $g(0)=0$ such that $$g(\|x-y\|)\leq\phi(x,y),$$ for all $x,y\in B_{r}(0)=\{z\in E:\|z\|\leq r\}$.
\end{lm}
\begin{lm}\cite{Cho}\label{lem15}
Let $E$ be a uniformly convex Banach space. Then there exists a continuous strictly increasing convex function $g:[0, 2r]\rightarrow [0, \infty)$, $g(0)=0$ such that $$\|\lambda x+\mu y+\gamma z\|^{2}\leq \lambda\|x\|^{2}+\mu\|y\|^{2}+\gamma\|z\|^{2}-\lambda\mu g(\|x-y\|),$$
for all $x,y,z\in B_{r}(0)=\{z\in E:\|z\|\leq r\}$ and all $\lambda, \mu, \gamma\in [0, 1]$ with $\lambda+\mu+\gamma=1$.
\end{lm}
\begin{lm}\cite{Kamimura}\label{lem5}
Let $E$ be a uniformly convex and smooth Banach space and let $\{x^{k}\}$ and $\{y^{k}\}$ be two sequences of $E$. If $\phi(x^{k}, y^{k})\rightarrow 0$ and either $\{x^{k}\}$ or $\{y^{k}\}$ is bounded, then $x^{k}-y^{k}\rightarrow 0.$
\end{lm}
\begin{lm}\label{lem10}\cite{ka}
Let $\{\alpha_{n}\}$ and $\{\beta_{n}\}$ be two positive and bounded sequences in $\mathbb{R}$, then
\begin{equation*}
\begin{aligned}
(\liminf_{n\rightarrow \infty}\alpha_{n})\times(\liminf_{n\rightarrow\infty}\beta_{n})&\leq\liminf_{n\rightarrow\infty}(\alpha_{n}\beta_{n})\\
&\leq(\liminf_{n\rightarrow\infty}\alpha_{n})\times(\limsup_{n\rightarrow\infty}\beta_{n})\\
&\leq\limsup_{n\rightarrow\infty}(\alpha_{n}\beta_{n})\\
&\leq(\limsup_{n\rightarrow\infty}\alpha_{n})\times(\limsup_{n\rightarrow\infty}\beta_{n}).
\end{aligned}
\end{equation*}
\end{lm}
%%%%%%%%%%%%%%%%%%%%%%%%%%%%%%%%%%%%%%%%%%%%%%%%%%%%%%%%%%%%%%%%%%%%%%%%%%%%%%%%%%%%%%%%%%%%%%%%%%%%%%%
\section{An Extragradient Algorithm}
In this section, we present an algorithm for finding a solution of the ($EP$) which
is also the common element of the set of fixed points of a relatively nonexpansive mapping.\\
Here, we assume that bifunction $f:C\times C\rightarrow\mathbb{R}$ satisfies in following conditions which $C$ is nonempty, convex and closed subset of uniformly convex and uniformly
smooth Banach space $E$,
 \begin{enumerate}
\item[(A1)]$f(x,x)=0$ for all $x\in C,$
\item[(A2)]$f$ is pseudomonotone on $C$, i.e., $f(x,y)\geq0\Longrightarrow f(y,x)\leq0$ for all $x,y\in C,$
\item[(A3)]$f$ is jointly weakly continuous on $C\times C$, i.e., if $x,y\in C$ and $\{x_{n}\}$ and $\{y_{n}\}$ are two sequences in $C$ converging weakly to $x$ and $y$, respectively, then $f(x_{n},y_{n})\rightarrow f(x,y),$
\item[(A4)]$f(x,.)$ is convex, lower semicontinuous and subdifferentiable on $C$ for every $x\in C,$
\item[(A5)]$f$ satisfies $\phi$-Lipschitz-type condition: $\exists c_{1}>0, \exists c_{2}>0,$ such that for every $x,y,z\in C$
\begin{equation}\label{equa17}
 f(x,y)+f(y,z)\geq f(x,z)-c_{1}\phi(y,x)-c_{2}\phi(z,y).
 \end{equation} \end{enumerate}
It is easy to see that if $f$ satisfies the properties $(A_{1})-(A_{4})$, then the set $E(f)$ of solutions of an equilibrium problem is closed and convex. Indeed, when $E$ is a Hilbert space, $\phi$-Lipschitz-type condition reduces to Lipschitz-type condition (\ref{equ.2}).\\
Throughout the paper $S$ is a relatively nonexpansive self-mapping of $C$.
\par
%%%%%%%%%%%%%%%%%%%%%%%%%%%%%%%%%%%%%%%%%%%%%%%%%%%%%%%%%%%%%%%%%%%%%%%%%%%%%%%%%%%%%%%%%%%%%%%%%%%%%%%%
{\bf Algorithm $1$}
\begin{description}
\item[Step 0.] Suppose that $\{\alpha_{n}\}\subseteq [a,e]$ for some $0<a<e<1$, $\{\beta_{n}\}\subseteq [d,b]$ for some $0<d<b<1$ and $\{\lambda_{n}\}\subseteq (0,1]$ such that $\{\lambda_{n}\}\subseteq [\lambda_{min},\lambda_{max}]$, where
 $$0<\lambda_{min}\leq\lambda_{max}<\min\{\frac{1}{2c_{1}},\frac{1}{2c_{2}}\}.$$
\item[Step 1.] Let $x_{0}\in C$. Set n=0.\\
\item[Step 2.] Compute $y_{n}$ and $x_{n}$, such that
$$y_{n}=\arg\min_{y\in C}\{\lambda_{n}f(x_{n},y)+\frac{1}{2}\phi(y,x_{n})\},$$
$$z_{n}=\arg\min_{y\in C}\{\lambda_{n}f(y_{n},y)+\frac{1}{2}\phi(y,x_{n})\}.$$
\item[Step 3.] Compute $t_{n}=J^{-1}(\alpha_{n}Jx_{n}+(1-\alpha_{n})(\beta_{n}Jz_{n}+(1-\beta_{n})JSz_{n}))$.\\
If $y_{n}=x_{n}$ and $t_{n}=x_{n}$, then $x_{n}\in E(f)\cap F(S)$ and go to step 4.\\
\item[Step 4.] Compute $x_{n+1}=R_{C_{n}\cap D_{n}}x_{0}$, where $R_{C_{n}\cap D_{n}}$ is the sunny generalized nonexpansive retraction from $E$
onto $C_{n}\cap D_{n}$ and
$$C_{n}=\{z\in C: \phi(z,t_{n})\leq\phi(z,x_{n})\},$$
$$D_{n}=\{z\in C: \langle Jx_{n}-Jz,x_{0}-x_{n}\rangle\geq 0\}.$$
\item[Step 5.] set $n:=n+1$ and go to Step 2.
\end{description}
Before proving the strong convergence of the iterates generated by Algorithm 1, we prove following lemmas.
\begin{lm}\label{lem1}
 For every $x^{*}\in E(f)$ and $n\in \mathbb{N}$, we obtain
\begin{enumerate}
\item[(i)] $\langle Jx_{n}-Jy_{n},y-y_{n}\rangle\leq\lambda_{n}f(x_{n},y)-\lambda_{n}f(x_{n},y_{n}),\:\:\: \ \forall y\in C,$
\item[(ii)] $\phi(x^{*},z_{n})\leq\phi(x^{*},x_{n})-(1-2\lambda_{n}c_{1})\phi(y_{n},x_{n})-(1-2\lambda_{n}c_{2})\phi(z_{n},y_{n})$.
\end{enumerate}
 \end{lm}
 \begin{proof}
By the condition ($A4$) for $f(x,.)$ and from lemmas \ref{lem7} and \ref{lem8}, we obtain
$$z_{n}=\arg\min_{y\in C}\{\lambda_{n}f(y_{n},y)+\frac{1}{2}\phi(y,x_{n})\}$$
if and only if
$$0\in \lambda_{n}\partial_{2}f(y_{n},z_{n})+\frac{1}{2}\nabla_{1}\phi(z_{n},x_{n})+N_{C}(z_{n}).$$
This implies that $w\in \partial_{2}f(y_{n},z_{n})$ and $\overline{w}\in N_{C}(z_{n})$ exist such that
\begin{equation}\label{equa18}
0=\lambda_{n}w+Jz_{n}-Jx_{n}+\overline{w},
\end{equation}
so, from definition of $\partial_{2}f(y_{n},z_{n})$, we obtain
$$\langle w, y-z_{n}\rangle\leq f(y_{n},y)-f(y_{n},z_{n}),$$
for all $y\in C$. Set $y=x^{*}$, we have
$$\langle w, x^{*}-z_{n}\rangle\leq f(y_{n},x^{*})-f(y_{n},z_{n}).$$
So, by definition of the $N_{C}(z_{n})$ and equality (\ref{equa18}), we get
\begin{equation}\label{eq40} \lambda_{n}\langle w, z_{n}-y\rangle\leq \langle Jz_{n}-Jx_{n},y-z_{n}\rangle,
\end{equation} for all $y\in C$. Put $y=x^{*}$ in inequality \ref{eq40}, we have
\begin{equation}\label{equa1}
\langle Jz_{n}-Jx_{n},x^{*}-z_{n}\rangle\geq\lambda_{n}\{f(y_{n},z_{n})-f(y_{n},x^{*})\}
\geq\lambda_{n}f(y_{n},z_{n}),
\end{equation}
since $f(x^{*},y_{n})\geq0$ and $f$ is pseudomonotone on $C$. Replacing $x, y$ and $z$ by $x_{n}$, $y_{n}$ and $z_{n}$ in inequality (\ref{equa17}), respectively, we get
\begin{equation}\label{equa2}
f(y_{n},z_{n})\geq f(x_{n},z_{n})-f(x_{n},y_{n})-c_{1}\phi(y_{n},x_{n})-c_{2}\phi(z_{n},y_{n}).
\end{equation}
In a similar way, since $y_{n}=\arg\min_{y\in C}\{\lambda_{n}f(x_{n},y)+\frac{1}{2}\phi(y,x_{n})\}$, we have
$$\langle Jx_{n}-Jy_{n},y-y_{n}\rangle\leq\lambda_{n}\{f(x_{n},y)-f(x_{n},y_{n})\},$$
for all $y\in C$, hence ($i$) is proved. Let $y=z_{n}$ in above inequality, we obtain
\begin{equation}\label{equa3}
\langle Jx_{n}-Jy_{n},z_{n}-y_{n}\rangle\leq\lambda_{n}\{f(x_{n},z_{n})-f(x_{n},y_{n})\}.
\end{equation}
Combining inequalities (\ref{equa1}), (\ref{equa2}) and (\ref{equa3}), we get
\begin{equation}\label{equa19}
2\langle Jz_{n}-Jx_{n},x^{*}-z_{n}\rangle\geq 2\langle Jy_{n}-Jx_{n},y_{n}-z_{n}\rangle-2\lambda_{n}c_{1}\phi(y_{n},x_{n})-2\lambda_{n}c_{2}\phi(z_{n},y_{n}).
\end{equation}
From inequality (\ref{equa19}) and (\ref{eq36}), we have
\begin{equation*}
\begin{aligned}
\phi(x^{*},x_{n})-\phi(x^{*},z_{n})\geq\phi(z_{n},y_{n})+\phi(y_{n},x_{n})
-2\lambda_{n}c_{1}\phi(y_{n},x_{n})-2\lambda_{n}c_{2}\phi(z_{n},y_{n}).
\end{aligned}
\end{equation*}
Hence, ($ii$) is proved.
\end{proof}
\begin{rk}
In a real Hilbert space $E$, Lemma \ref{lem1} is reduced to Lemma
$3.1$ in \cite{Anh}.
\end{rk}
\begin{lm}
In Algorithm $1$, we obtain the unique optimal solutions $y_{n}$ and $z_{n}$.
\end{lm}
\begin{proof}
Let $y_{n},\acute{y}_{n}\in\arg\min_{y\in C}\{\lambda_{n}f(x_{n},y)+\frac{1}{2}\phi(y,x_{n})\}$, then using Lemma \ref{lem1}($i$), we have
\begin{equation}\label{equa22}
\langle Jx_{n}-Jy_{n},y-y_{n}\rangle\leq\lambda_{n}f(x_{n},y)-\lambda_{n}f(x_{n},y_{n}),\:\:\: \ \forall y\in C,
\end{equation}
\begin{equation}\label{equa23}
\langle Jx_{n}-J\acute{y}_{n},y-\acute{y}_{n}\rangle\leq\lambda_{n}f(x_{n},y)-\lambda_{n}f(x_{n},\acute{y}_{n}),\:\:\: \ \forall y\in C.
\end{equation}
Set $y=\acute{y}_{n}$ in inequality (\ref{equa22}) and $y=y_{n}$ in inequality (\ref{equa23}). Hence, we get
$$\langle J\acute{y}_{n}-Jy_{n},\acute{y}_{n}-y_{n}\rangle\leq0.$$
Since $J$ is monotone and one-one, we obtain $y_{n}=\acute{y}_{n}.$ In a similar way, also $z_{n}$ is unique.
\end{proof}
\begin{rk}
If $E$ is a real Hilbert space, then Algorithm $1$ is the same
Extragradient Algorithm in \cite{Vuo} provided that the sequence
$\{\alpha_{n}\}$ satisfies the conditions of Step $0$ of Algorithm
$1$.
%same Algorithm $1$ in \cite{Vuo} (extragradient algorithm)
\end{rk}
\begin{lm}\label{lem3}
For every $x^{*}\in E(f)\cap F(S)$ and $n\in \mathbb{N}$, we get
$$\phi(x^{*},t_{n})\leq\phi(x^{*},x_{n}).$$
\end{lm}
\begin{proof}
From Lemma \ref{lem1}, by the convexity of $\|.\|^{2}$ and by the definition of functions $\phi$ and $S$, we have
\begin{equation}\label{equa37}
\begin{aligned}
\phi(x^{*},t_{n})&=\phi(x^{*},J^{-1}(\alpha_{n}Jx_{n}+(1-\alpha_{n})(\beta_{n}Jz_{n}+(1-\beta_{n})JSz_{n})))\\
&=\|x^{*}\|^{2}+\|\alpha_{n}Jx_{n}+(1-\alpha_{n})(\beta_{n}Jz_{n}+(1-\beta_{n})JSz_{n})\|\\
&-2\langle x^{*},\alpha_{n}Jx_{n}+(1-\alpha_{n})(\beta_{n}Jz_{n}+(1-\beta_{n})JSz_{n})\rangle\\
&\leq\|x^{*}\|^{2}+\alpha_{n}\|x_{n}\|^{2}+(1-\alpha_{n})\beta_{n}\|z_{n}\|^{2}+(1-\alpha_{n})(1-\beta_{n})\|Sz_{n}\|^{2}\\
&-2\alpha_{n}\langle x^{*},Jx_{n}\rangle-2(1-\alpha_{n})\beta_{n}\langle x^{*},Jz_{n}\rangle-2(1-\alpha_{n})(1-\beta_{n})\langle x^{*},JSz_{n}\rangle\\
&=\alpha_{n}\phi(x^{*},x_{n})+(1-\alpha_{n})\beta_{n}\phi(x^{*},z_{n})+(1-\alpha_{n})(1-\beta_{n})\phi(x^{*},Sz_{n})\\
&\leq\phi(x^{*},x_{n}).
\end{aligned}
\end{equation}
\end{proof}
We examine the stopping condition in the following lemma.
\begin{lm}\label{lem6}
Let $y_{n}=x_{n}$, then $x_{n}\in E(f)$. If $y_{n}=x_{n}$ and $t_{n}=x_{n}$, then $x_{n}\in E(f)\cap F(S)$.
\end{lm}
\begin{proof}
Suppose $y_{n}=x_{n}$, then by the definition of $y_{n}$, condition ($A1$), property of $\phi$ (\ref{equa21}) and since $0<\lambda_{min}\leq\lambda_{n}\leq\lambda_{max}\leq1$ , we have
$$0\leq \lambda_{n}f(x_{n},y_{n})+\frac{1}{2}\phi(y_{n},x_{n})\leq f(x_{n},y)+\frac{1}{2}\phi(y,x_{n}),$$
for all $y\in C$. Set $\phi(y,x_{n})=L(x_{n},y)$, hence Lemma \ref{lem2} implies that $x_{n}\in E(f)$.\\
Let $y_{n}=x_{n}$ and $t_{n}=x_{n}$, we have that $z_{n}=x_{n}$ and since $J^{-1}$ is one-one, we get
$$Jx_{n}=\alpha_{n}Jx_{n}+(1-\alpha_{n})(\beta_{n}Jx_{n}+(1-\beta_{n})JSx_{n}).$$
Since $1-\alpha_{n}>0$ and $1-\beta_{n}>0$, it follows that $Jx_{n}=JSx_{n}$ and since $J$ is one-one, we get $x_{n}=Sx_{n}$. So $x_{n}\in F(S)$
\end{proof}
\begin{rk}
In a real Hilbert space $E$, Lemma \ref{lem6} is the same
Proposition $3.5$ in \cite{Vuo} with different proof, provided
that the sequence $\{\alpha_{n}\}$ satisfies the conditions of
Step $0$ of Algorithm $1$.
\end{rk}
\begin{thm}\label{eq25}
Let $C$ be a nonempty closed convex subset of a uniformly
convex, uniformly smooth Banach space $E$.
 Assume that $f:C\times
C\rightarrow \mathbb{R}$ is a bifunction which satisfies
conditions $(A1)-(A5)$ and $S:C\rightarrow C$
is a relatively nonexpansive mapping such that $$\Omega:=E(f)\cap F(S)\neq\emptyset.$$
%Assume that $R_{C}$ is the sunny generalized nonexpansive retraction from $E$ onto $C$.
Then sequences
$\{x_{n}\}_{n=1}^{\infty}$, $\{y_{n}\}_{n=1}^{\infty}$, $\{z_{n}\}_{n=1}^{\infty}$ and $\{t_{n}\}_{n=1}^{\infty}$ generated
by Algorithm $1$ converge strongly to the some solution $u^{*}\in
\Omega$, where $u^{*}=R_{\Omega}{x_{0}}$, and $R_{\Omega}$ is sunny generalized nonexpansive retraction from $E$ onto $\Omega$.
\end{thm}
\begin{proof}
At First, using induction we show that $\Omega\subseteq C_{n}\cap D_{n}$ for all $n\geq 0$. Let $x^{*}\in\Omega$, from Lemma \ref{lem3}, we get $\Omega\subseteq C_{n}$ for all $n\geq 0$. Now, we show that $\Omega\subseteq D_{n}$ for all $n\geq0$. It is clear that $\Omega\subseteq D_{0}$. Suppose that $\Omega\subseteq D_{n}$, i.e $\langle Jx_{n}-Jx^{*},x_{0}-x_{n}\rangle\geq0,$ for all $x^{*}\in \Omega$. Since $x_{n+1}=R_{C_{n}\cap D_{n}}x_{0}$, using Lemma \ref{lem4}, we get $\langle Jx_{n+1}-Jz,x_{0}-x_{n+1}\rangle\geq0,$ for all $z\in C_{n}\cap D_{n}$. This implies that $x^{*}\in D_{n+1}$. Therefore $\Omega\subseteq D_{n+1}$.
\par
Let $x^{*}\in\Omega\subseteq D_{n+1}$. Since $x_{n+1}\in D_{n}$,
using successively equality (\ref{eq36}),  it is easy to see that
the $\{\phi(x_{0},x_{n})\}$ is increasing and bounded from above
by $\phi(x_{0},x^{*})$, so
 $\lim\limits_{n\rightarrow\infty}\phi(x_{0},x_{n})$ exists.
 This yields that $\{\phi(x_{0},x_{n})\}$ is bounded. From inequality (\ref{eq14}),
 we know that $\{x_{n}\}$ is bounded. It is clear that $\lim_{n\rightarrow\infty}\phi(x_{n},x_{n+1})=0$,
 so Lemma \ref{lem5} implies that $\lim\limits_{n\rightarrow\infty}\|x_{n}-x_{n+1}\|=0$ and
 therefore $\{x_{n}\}$ converges strongly to $\bar{x}\in C.$ Since $J$ is uniformly norm-to-norm
 continuous on bounded sets, from equality (\ref{equa6}),
 we obtain that $\lim_{n\rightarrow\infty}\phi(x_{n+1},x_{n})=0$,
 and $\lim\limits_{n\rightarrow\infty}\phi(\nu,x_{n})$ exists for all $\nu \in C$. Since $x_{n+1}\in C_{n}$,
 we have $\lim\limits_{n\rightarrow\infty}\phi(x_{n+1},t_{n})=0$ and from Lemma \ref{lem5},
 we deduce that $\lim\limits_{n\rightarrow\infty}\|x_{n+1}-t_{n}\|=0$,
 thus $\lim\limits_{n\rightarrow\infty}\|x_{n}-t_{n}\|=0$ which implies that $\{t_{n}\}$
  converges strongly to $\bar{x}$. Using norm-to-norm continuity of $J$
  on bounded sets, we conclude that $\lim\limits_{n\rightarrow\infty}\|Jx_{n}-Jt_{n}\|=0$ and therefore
\begin{equation}\label{equa9}
\lim\limits_{n\rightarrow\infty}\phi(x^{*},x_{n})=\lim\limits_{n\rightarrow\infty}\phi(x^{*},t_{n}).
\end{equation}
Using Lemma \ref{lem1} $(ii)$, we obtain $\phi(x^{*},z_{n})\leq\phi(x^{*},x_{n})$. From inequality (\ref{eq14}) and the definition of $S$, we derive that $\{z_{n}\}$ and $\{Sz_{n}\}$ are bounded.
Let $r_{1}=\sup_{n\geq0}\{\|x_{n}\|,\|z_{n}\|\}$ and\linebreak $r_{2}=\sup_{n\geq0}\{\|z_{n}\|,\|Sz_{n}\|\}$. So, by Lemma \ref{lem15}, there exists a continuous, strictly increasing and convex function $g_{1}:[0,2r_{1}]\rightarrow\mathbb{R}$ with $g_{1}(0)=0$ such that for all $x^{*}\in \Omega$, we get
\begin{equation}
\begin{aligned}
\phi(x^{*},t_{n})&=\phi(x^{*},J^{-1}(\alpha_{n}Jx_{n}+(1-\alpha_{n})(\beta_{n}Jz_{n}+(1-\beta_{n})JSz_{n})))\\
&\leq\|x^{*}\|^{2}+\alpha_{n}\|x_{n}\|^{2}+(1-\alpha_{n})\beta_{n}\|z_{n}\|^{2}+(1-\alpha_{n})(1-\beta_{n})\|Sz_{n}\|^{2}
\\
&\hspace{.5cm}-2\alpha_{n}\langle
x^{*},Jx_{n}\rangle-2(1-\alpha_{n})\beta_{n}\langle x^{*},Jz_{n}
\rangle-2(1-\alpha_{n})(1-\beta_{n})\langle x^{*},JSz_{n}\rangle\\
&\hspace{.5cm}-\alpha_{n}(1-\alpha_{n})\beta_{n}g_{1}(\|Jz_{n}-Jx_{n}\|)\\
&\leq\phi(x^{*},x_{n})-\alpha_{n}(1-\alpha_{n})\beta_{n}g_{1}(\|Jz_{n}-Jx_{n}\|),
\end{aligned}
\end{equation}
and using the same argument, there exists a continuous, strictly increasing and convex function $g_{2}:[0,2r_{2}]\rightarrow\mathbb{R}$ with $g_{2}(0)=0$ such that for all $x^{*}\in \Omega$, we have
$$\phi(x^{*},t_{n})\leq\phi(x^{*},x_{n})-(1-\alpha_{n})^{2}\beta_{n}(1-\beta_{n})g_{2}(\|Jz_{n}-JSz_{n}\|)=0,$$
which imply
\begin{equation}\label{equa8}
\alpha_{n}(1-\alpha_{n})\beta_{n}g_{1}(\|Jz_{n}-Jx_{n}\|)\leq\phi(x^{*},x_{n})-\phi(x^{*},t_{n}),
\end{equation}
\begin{equation}\label{equa11}
(1-\alpha_{n})^{2}\beta_{n}(1-\beta_{n})g_{2}(\|Jz_{n}-JSz_{n}\|)\leq\phi(x^{*},x_{n})-\phi(x^{*},t_{n}).
\end{equation}
By letting $n\rightarrow\infty$ in inequalities (\ref{equa8}) and
(\ref{equa11}),
 using Lemma \ref{lem10} and equality (\ref{equa9}), we obtain
%\begin{equation*}
%\liminf_{n\rightarrow\infty}\alpha_{n}(1-\alpha_{n})\beta_{n}>0\quad \& \quad
%\liminf_{n\rightarrow\infty} (1-\alpha_{n})\beta_{n}(1-\beta_{n})>0,
%\end{equation*}
%it following from Lemma \ref{lem10} and inequality (\ref{equa9}) that
\begin{equation*}
\lim\limits_{n\rightarrow\infty}g_{1}(\|Jz_{n}-Jx_{n}\|)=0\quad \& \quad
\lim\limits_{n\rightarrow\infty}g_{2}(\|Jz_{n}-JSz_{n}\|)=0.
\end{equation*}
From the properties of $g_{1}$ and $g_{2}$, we have
\begin{equation}\label{equa33}
\lim\limits_{n\rightarrow\infty}\|Jz_{n}-Jx_{n}\|=0\quad \& \quad
\lim\limits_{n\rightarrow\infty}\|Jz_{n}-JSz_{n}\|=0.
\end{equation}
So from (\ref{eq18}), we obtain
\begin{equation}\label{equa12}
\lim\limits_{n\rightarrow\infty}\|z_{n}-x_{n}\|=0\quad \& \quad
\lim\limits_{n\rightarrow\infty}\|z_{n}-Sz_{n}\|=0,
\end{equation}
since $J^{-1}$ is uniformly norm-to-norm continuous on bounded sets.
By the same reason as in the proof of (\ref{equa9}), we can conclude
from (\ref{equa33}) and (\ref{equa12}) that
\begin{equation}\label{equa34}
\lim_{n\rightarrow\infty}\phi(x^{*},x_{n})=\lim_{n\rightarrow\infty}\phi(x^{*},z_{n}),
\end{equation}
for all $x^{*} \in \Omega$. Using Lemma \ref{lem1} $(ii)$, we have
\begin{equation}\label{equa35}
(1-2\lambda_{n}c_{1})\phi(y_{n},x_{n})+(1-2\lambda_{n}c_{2})\phi(z_{n},y_{n})\leq\phi(x^{*},x_{n})-\phi(x^{*},z_{n}).
\end{equation}
for all $x^{*} \in \Omega$. Taking the limits as
$n\rightarrow\infty$ in inequality (\ref{equa35}), using equality
(\ref{equa34}), we get
\begin{equation}\label{equa10}
\lim\limits_{n\rightarrow\infty}\phi(y_{n},x_{n})=0\quad \& \quad
\lim\limits_{n\rightarrow\infty}\phi(z_{n},y_{n})=0.
\end{equation}
Since $\{x_{n}\}$ and $\{z_{n}\}$ are bounded, it follows from Lemma \ref{lem5} that
\begin{equation*}
\lim\limits_{n\rightarrow\infty}\|y_{n}-x_{n}\|=0\quad \& \quad
\lim\limits_{n\rightarrow\infty}\|z_{n}-y_{n}\|=0,
\end{equation*}
which imply $\{y_{n}\}$ and $\{z_{n}\}$ converges strongly to
$\bar{x}\in C.$
\par
Now, we prove that $\bar{x}\in E(f)$. It follows from the definition
of $y_{n}$ that
\begin{equation}\label{equa36}
\lambda_{n}f(x_{n},y_{n})+\frac{1}{2}\phi(y_{n},x_{n})\leq
\lambda_{n}f(x_{n},y)+\frac{1}{2}\phi(y,x_{n})
\end{equation}
for all $y\in C$. By letting $n\rightarrow\infty$ in inequality
(\ref{equa36}), it follows from equality (\ref{equa10}), conditions
$(A1)$ and $(A3)$ and uniformly norm-to-norm continuity of $J$  on
bounded sets that
$$0\leq f(\bar{x},y)+\phi(y,\bar{x}),$$
because of $0<\lambda_{min}\leq\lambda_{n}\leq\lambda_{max}\leq1$.
Hence, letting
 $\phi(y,\bar{x})=L(\bar{x},y)$, Lemma \ref{lem2} implies that $\bar{x}\in E(f)$.
 \par
Now, since $z_{n}\rightharpoonup\bar{x}$, from (\ref{equa12}), we get $\bar{x}\in \hat{F}(S)$. So,
 using the definition of $S$, we have $\bar{x}\in F(S).$ Setting $z=u^{*}$ in Lemma \ref{lem4}, since $x_{n+1}=R_{C_{n}\cap D_{n}}x_{0}$ and $\phi$ is continuous respect to the first argument, we obtain
$$\phi(u^{*},x_{0})\geq\lim\limits_{n\rightarrow\infty}\phi(x_{n+1},x_{0})=\lim\limits_{n\rightarrow\infty}\phi(x_{n},x_{0})=\phi(\bar{x},x_{0}),$$
also, using Lemma \ref{eq38}, we have
$$\phi(u^{*},x_{0})\leq\phi(y,x_{0}),$$
for all $y\in \Omega,$ because of $u^{*}=R_{\Omega}x_{0}$.
Therefore $\bar{x}=u^{*}$ and consequently the sequences
$\{x_{n}\}_{n=1}^{\infty}$, $\{y_{n}\}_{n=1}^{\infty}$,
$\{z_{n}\}_{n=1}^{\infty}$ and $\{t_{n}\}_{n=1}^{\infty}$ converge
strongly to $R_{\Omega}{x_{0}}$.
\end{proof}
\begin{rk}
If $E$ is a real Hilbert space, then Theorem \ref{eq25} is the
same Theorem $3.1$ in \cite{Vuo} for a nonexpansive mapping $S$
with different proof, provided that the sequence $\{\alpha_{n}\}$
satisfies the conditions of Step $0$ of Algorithm $1$.
\end{rk}
\section{A Linesearch Algorithm}
As we see in the previous section, $\phi$-Lipschitz-type condition
$(A 5)$ depends on two positive parameters $c_{1}$ and $c_{2}$. In
some cases, these parameters  are unknown or difficult to
approximate. To avoid this difficulty, using linesearch method, we
modify Extragradient Algorithm. We prove strong convergence of
this new algorithm without assuming the $\phi$-Lipschitz-type
condition. linesearch method has a good efficiency in numerical
tests.

\par
Here, we assume that bifunction $f:\Delta\times \Delta\rightarrow\mathbb{R}$ satisfies in conditions $(A1)$, $(A2)$ and $(A4)$ and also in following condition which $C$ is nonempty, convex and closed of $2$-uniformly convex, uniformly
smooth Banach space $E$ and $\Delta$ is an open convex set containing $C$,\\
($A3^{*}$) $f$ is jointly weakly continuous on $\Delta\times\Delta$, i.e., if $x,y\in C$ and $\{x_{n}\}$ and $\{y_{n}\}$ are two sequences in $\Delta$ converging weakly to $x$ and $y$, respectively, then $f(x_{n},y_{n})\rightarrow f(x,y).$\\
\par
{\bf Algorithm $2$}
\begin{description}
\item[Step 0.] Let $\alpha\in(0,1)$ , $\gamma \in (0,1)$ and suppose that
$\{\alpha_{n}\}\subseteq [a,e]$ for some $0<a < e<1$,
$\{\beta_{n}\}\subseteq [d,b]$ for some $0<d < b<1$,
$\{\lambda_{n}\}\subseteq [\lambda,1]$, where
 $0<\lambda\leq1$ and $0<\nu<\frac{c^{2}}{2}$, where $\frac{1}{c}$ ($0<c\leq 1$) is the $2-$uniformly convexity constant of $E$.
 \item[Step 1.] Let $x_{0}\in C$. set $n=0.$
 \item[Step 2.] Obtain the unique optimal solution $y_{n}$ by Solving the following convex problem
 \begin{equation}\label{equa38}
 \min_{y\in C}\{\lambda_{n}f(x_{n},y)+\dfrac{1}{2}\phi(x_{n},y)\}
 \end{equation}
 \item[Step 3.] If $y_{n}=x_{n}$, then set $z_{n}=x_{n}$. Otherwise
 \begin{description}
 \item[Step 3.1.] Find $m$ the smallest nonnegative integer such that
 \begin{equation}\label{equa41}
\begin{cases}
f(z_{n,m},x_{n})-f(z_{n,m},y_{n})\geq\frac{\alpha}{2\lambda_{n}}\phi(y_{n},x_{n})\: where\\
z_{n,m}=(1-\gamma^{m})x_{n}+\gamma^{m}y_{n}.
\end{cases}
\end{equation}
 \item[Step 3.2.] Set $\rho_{n}=\gamma^{m}$, $z_{n}=z_{n,m}$ and go to Step 4.
  \end{description}
  \item[Step 4.] Choose $g_{n}\in\partial _{2} f(z_{n},x_{n})$ and compute $w_{n}=R_{C}J^{-1}(Jx_{n}-\sigma_{n}g_{n})$. If $y_{n}\neq x_{n}$, then $\sigma_{n}=\dfrac{\nu f(z_{n},x_{n})}{\Vert g_{n}\Vert^{2}}$ and $\sigma_{n}=0$ otherwise.
 \item[Step 5.] Compute $t_{n}=J^{-1}(\alpha_{n}Jx_{n}+(1-\alpha_{n})(\beta_{n}Jw_{n}+(1-\beta_{n})JSw_{n})).$
 If $y_{n}=x_{n}$ and $t_{n}=x_{n},$ then STOP: $x_{n}\in E(f)\cap F(S)$. Otherwise, go to Step 6.
 \item[Step 6.] Compute $x_{n+1}=R_{C_{n}\cap D_{n}}x_{0}$, where
 $$C_{n}=\{z\in C: \phi(z,t_{n})\leq\phi(z,x_{n})\},$$
$$D_{n}=\{z\in C: \langle Jx_{n}-Jz,x_{0}-x_{n}\rangle\geq 0\}.$$
 \item[Step 7.] Set n:=n+1, and go to Step 2.
 \end{description}
 The following lemma shows that linesearch corresponding to $x_{n}$ and $y_{n}$ (Step 3.1) is well defined.
\begin{lm}\label{lem18}
  Assume that  $y_{n}=x_{n}$ for some $n\in\mathbb{N}$. Then
\begin{enumerate}
\item[(i)]
There exists a nonnegative integer $m$ such that the inequality in (\ref{equa41}) holds.
\item[(ii)]
$f(z_{n},x_{n})>0.$
\item[(iii)]
$0\notin \partial_{2}f(z_{n},x_{n}).$
 \end{enumerate}
 \end{lm}
 \begin{proof}
 %The proof of lemma is similar to proof of \cite{ Tr, $Lemma 4.5 and$ Vuo, $proposition 4.1$}.
 Suppose that $n\in \mathbb{N}$. Assume towards a contradiction that for each nonnegative integer $m$,
 \begin{equation}\label{equa15}
f(z_{n,m},x_{n})-f(z_{n,m},y_{n})<\frac{\alpha}{2\lambda_{n}}\phi(y_{n},x_{n}),
 \end{equation}
 where $z_{n,m}=(1-\gamma^{m})x_{n}+\gamma^{m}y_{n}.$
It is easy to see that $z_{n,m}\rightarrow x_{n}$ as $m\rightarrow\infty$. Using condition $(A3^{*})$, we obtain
 \begin{equation}
  f(z_{n,m},x_{n})\rightarrow f(x_{n},x_{n})\quad \& \quad
   f(z_{n,m},y_{n})\rightarrow f(x_{n},y_{n}).
  \end{equation}
  Since $f(x_{n},x_{n})=0$, letting $m\rightarrow\infty$ in inequality (\ref{equa15}), we get
\begin{equation}\label{equa16}
0\leq f(x_{n},y_{n})+\frac{\alpha}{2\lambda_{n}}\phi(y_{n},x_{n}).
\end{equation}
% Then $\{z_{n,m}\}_{m}$ converges strongly to $x_{n}$ as $m\rightarrow\infty$, and thus also weakly. Since $f(.,x)$ is weakly continuous on an open set $\Delta\supset C$ for every $x\in\Delta$, if follows that  So, taking the limit on $m$ in (\ref{equa15}), yields
Because of  $y_{n}$ is a solution of the convex optimization problem (\ref{equa38}), we deduce
 $$\lambda_{n}f(x_{n},y)+\frac{1}{2}\phi(y,x_{n})\geq\lambda_{n}f(x_{n},y_{n})+\frac{1}{2}\phi(y_{n},x_{n}),$$
 for all $y\in C$. If $y=x_{n}$, then
\begin{equation}\label{equa40}
\lambda_{n}f(x_{n},y_{n})+\frac{1}{2}\phi(y_{n},x_{n})\leq 0.
\end{equation}
 It follows from inequalities (\ref{equa16}) and (\ref{equa40}) that
 \begin{equation}\label{equa39}
f(x_{n},y_{n})+\frac{1}{2\lambda_{n}}\phi(y_{n},x_{n})\leq f(x_{n},y_{n})+\frac{\alpha}{2\lambda_{n}}\phi(y_{n},x_{n}).
\end{equation}
 Therefore from inequality (\ref{equa39}), we obtain
 $$\frac{1-\alpha}{2}\phi(y_{n},x_{n})\leq 0,$$
 since $\lambda_{n}\leq1$.
It follows from (\ref{equa21}) that $\phi(y_{n},x_{n})>0$, because of $y_{n}\neq x_{n}$.
Thus, $1-\alpha\leq0$ or $\alpha\geq 1$ where contradict the assumption $\alpha\in(0,1)$. So, $(i)$ is proved.
\par
 Now, we prove $(ii)$. Since $f$ is convex, we obtain
 \begin{equation}\label{equa42}
\rho_{n}f(z_{n},y_{n})+(1-\rho_{n})f(z_{n},x_{n})\geq f(z_{n},z_{n})=0.
\end{equation}
consequently from inequality (\ref{equa42}), we get
 $$f(z_{n},x_{n})\geq \rho_{n}[f(z_{n},x_{n})-f(z_{n},y_{n})]\geq\frac{\alpha\rho_{n}}{\lambda_{n}}\phi(y_{n},x_{n})>0,$$
because of $y_{n}\neq x_{n}$. Therefore $f(z_{n},x_{n})>0.$\\
The proof $(iii)$ can be found in \cite{Tr} (Lemma $4.5$).
\end{proof}
\begin{rk}
If $E$ is a real Hilbert space, then Lemma \ref{lem18} is reduced
to Proposition $4.1$ in \cite{Vuo} when $\alpha\in (0,1).$
%with different proof, when the sequence $\{\alpha_{n}\}$ satisfies in Step $0$ of Algorithm $2$.
\end{rk}
 We examine the stopping condition in the following lemma where its proof is similar to the proof Lemma \ref{lem6}.
\begin{lm}\label{lem17}
Let $y_{n}=x_{n}$, then $x_{n}\in E(f)$. If $y_{n}=x_{n}$ and $t_{n}=x_{n}$, then $w_{n}=x_{n}$ and $x_{n}\in E(f)\cap F(S)$.
\end{lm}
\begin{rk}
If $E$ is a real Hilbert space, then Lemma \ref{lem17} is the same
Proposition $4.2$ in \cite{Vuo} with different proof, provided
that the sequence $\{\alpha_{n}\}$ satisfies the conditions of
Step $0$ of Algorithm $2$.
\end{rk}
 %The proof of following lemma can be found in \cite{Vuo} (Proposition $4.3$).
 \begin{lm}\label{lem14}
  Suppose that $f:\Delta\times \Delta\rightarrow\mathbb{R}$ is a bifunction satisfying conditions $(A3^{*})$ and $(A4)$. Let $\{x_{n}\}$ and $\{z_{n}\}$ be two sequences in $\Delta$ such that Suppose  $x_{n} \rightarrow \bar{x}$ and $z_{n}\rightharpoonup \bar{z}$, where $\bar{x}, \bar{z} \in \Delta$. Then we have
   %such that converging strongly to $\bar{x}$ and converging weakly to  $\bar{z}$, then %for any $\varepsilon>0$, there exist $\eta>0$ and $n_{\varepsilon}\in\mathbb{N}$ such that
 $$\partial_{2}f(z_{n},x_{n})\subseteq\partial_{2}f(\bar{z},\bar{x}).$$
% for every $n\geq n_{\varepsilon}$, where $B$ denotes the closed unit ball in $E^{*}$.
 \end{lm}
 \begin{proof}
 Let $x^{*}\in \partial_{2}f(z_{n},x_{n})$, It follows from condition $(A4)$ and the definition of $\partial_{2}f$ that
 \begin{equation*}\label{equa51}
 f(z_{n},x)\geq f(z_{n},x_{n})+\langle x^{*},x-x_{n}\rangle,
 \end{equation*}
 for all $x \in \Delta$. Taking the limits as $n\rightarrow\infty$, using the condition $(A3^{*})$, we give
 $$f(\bar{z},x)\geq f(\bar{z},\bar{x})+\langle x^{*},x-\bar{x}\rangle,$$
 for all $x \in \Delta$. Hence, $x^{*}\in \partial_{2}f(\bar{z},\bar{x})$.
 \end{proof}
 Now, we prove the following proposition for Algorithm $2$, which have important role in the proof of main result in this section.
 \begin{prop}\label{lem9}
 For all $x^{*}\in E(f)\cap F(S)$ and all $n\in \mathbb{N}$, we get
 \begin{enumerate}
 \item[(i)]
$\phi(x^{*},w_{n})\leq\phi(x^{*},x_{n})-(\frac{2}{\nu}-\frac{4}{c^{2}})\sigma_{n}^{2}\|g_{n}\|^{2},$
\item[(ii)]
$\phi(x^{*},t_{n})\leq\phi(x^{*},x_{n})-(1-\alpha_{n})(\frac{2}{\nu}-\frac{4}{c^{2}})\sigma_{n}^{2}\|g_{n}\|^{2}.$
 \end{enumerate}
 \end{prop}
 \begin{proof}
Using Lemma \ref{lem4}, the definition of $V$ and inequality (\ref{eq5}), we have
 \begin{equation}
 \begin{aligned}
 \phi(x^{*},w_{n})&=\phi(x^{*},R_{C}J^{-1}(Jx_{n}-\sigma_{n}g_{n}))\\
 &\leq\phi(x^{*},J^{-1}(Jx_{n}-\sigma_{n}g_{n}))\\
 &=V(x^{*},Jx_{n}-\sigma_{n}g_{n})\\
 &\leq V(x^{*},Jx_{n}-\sigma_{n}g_{n}+\sigma_{n}g_{n})-2\langle J^{-1}(Jx_{n}-\sigma_{n}g_{n})-x^{*},\sigma_{n}g_{n}\rangle\\
 &=\phi(x^{*},x_{n})-2\langle J^{-1}(Jx_{n}-\sigma_{n}g_{n})-x^{*},\sigma_{n}g_{n}\rangle\\
 &=\phi(x^{*},x_{n})-2\sigma_{n}\langle x_{n}-x^{*},g_{n}\rangle+2\langle J^{-1}(Jx_{n}-\sigma_{n}g_{n})-x_{n},-\sigma_{n}g_{n}\rangle.
 \end{aligned}
 \end{equation}
 Since $g_{n}\in \partial_{2}f(z_{n},x_{n})$, we get
 \begin{equation*}
\langle x_{n}-x^{*},g_{n}\rangle\geq f(z_{n},x_{n})-f(z_{n},x^{*})\geq f(z_{n},x_{n})=\frac{\sigma_{n}\|g_{n}\|^{2}}{\nu}.
 \end{equation*}
 Therefore, we obtain
 \begin{equation}\label{equa24}
 -\frac{2}{\nu}\sigma_{n}^{2}\|g_{n}\|^{2}\geq -2\sigma_{n}\langle x_{n}-x^{*},g_{n}\rangle.
 \end{equation}
 On the other hand, from Lemma \ref{eq6}, we get
 \begin{equation}\label{equa25}
 \begin{aligned}
 2\langle J^{-1}(Jx_{n}-&\sigma_{n}g_{n})-x_{n},-\sigma_{n}g_{n}\rangle\\
 &=2\langle J^{-1}(Jx_{n}-\sigma_{n}g_{n})-J^{-1}(Jx_{n}), -\sigma_{n}g_{n}\rangle\\
 &\leq 2\|J(J^{-1}(Jx_{n}-\sigma_{n}g_{n}))-J(J^{-1}Jx_{n})\|\|\sigma_{n}g_{n}\|\\
 &\leq\frac{4}{c^{2}}\sigma_{n}^{2}\|g_{n}\|^{2}.
 \end{aligned}
 \end{equation}
Where $\frac{1}{c}(0< c\leq 1)$ is the $2$-uniformly convex constant of $E$. Thus, combining inequalities (\ref{equa24}) and (\ref{equa25}), we can derive $(i)$. A similar argument as in the proof of Lemma \ref{lem3} shows that
 \begin{equation*}
 \phi(x^{*},t_{n})\leq\alpha_{n}\phi(x^{*},x_{n})+(1-\alpha_{n})\phi(x^{*},w_{n}).
 \end{equation*}
 Using (i), we see that
 \begin{equation*}
 \begin{aligned}
 \phi(x^{*},t_{n}) &\leq\alpha_{n}\phi(x^{*},x_{n})+(1-\alpha_{n})\phi(x^{*},x_{n})-(1-\alpha_{n})(\frac{2}{\nu}-\frac{4}{c^{2}})\sigma_{n}^{2}\|g_{n}\|^{2}\\
 &=\phi(x^{*},x_{n})-(1-\alpha_{n})(\frac{2}{\nu}-\frac{4}{c^{2}})\sigma_{n}^{2}\|g_{n}\|^{2}.
 \end{aligned}
 \end{equation*}
Therefore $(ii)$ is proved.
 \end{proof}
 \begin{thm}\label{thm1}
%Let $C$ be a nonempty closed convex subset of a uniformly
%convex, uniformly smooth Banach space $E$.
 %Assume that $f:C\times
%C\rightarrow \mathbb{R}$ is a bifunction which satisfies
%conditions (A1)-(A4) and
Assume that $S:C\rightarrow C$
is a relatively nonexpansive mapping such that $$\Omega:=E(f)\cap F(S)\neq\emptyset.$$
Then the sequences
$\{x_{n}\}_{n=1}^{\infty}$, $\{y_{n}\}_{n=1}^{\infty}$, $\{z_{n}\}_{n=1}^{\infty}$, $\{w_{n}\}_{n=1}^{\infty}$ and $\{t_{n}\}_{n=1}^{\infty}$ generated
by Algorithm $2$ converge strongly to the some solution $u^{*}\in
\Omega$, where $u^{*}=R_{\Omega}{x_{0}}$ and $R_{\Omega}$ is the sunny generalized nonexpansive retraction from $E$ onto $\Omega$.
\end{thm}
 \begin{proof}
Let $x^{*}\in \Omega$. Similar to the proof of Theorem \ref{eq25}, we have
\begin{equation}
 \lim\limits_{n\rightarrow\infty}\|x_{n+1}-x_{n}\|=0\quad \& \quad
 \lim\limits_{n\rightarrow\infty}\|x_{n}-t_{n}\|=0,
 \end{equation}
 which imply that $\{x_{n}\}$ and consequently $\{t_{n}\}$ converge strongly to $\bar{x}\in C$ and
 \begin{equation}\label{equa43}
 \lim\limits_{n\rightarrow\infty}(\phi(x^{*},x_{n})-\phi(x^{*},t_{n}))=0.
 \end{equation}
 Since $(1-\alpha_{n})(\frac{2}{\nu}-\frac{4}{c^{2}})>0$, it follows from Lemma \ref{lem9} $(ii)$ that
 $\:\lim\limits_{n\rightarrow\infty}\sigma_{n}\|g_{n}\|=0.$
 \par
 Now, we prove that $\{y_{n}\}$, $\{z_{n}\}$ and $\{g_{n}\}$ are bounded and $f(z_{n},x_{n})\rightarrow 0$, as $n\rightarrow\infty$. Suppose that
 $$A(y)=\lambda_{n}f(x_{n},y)+\frac{1}{2}\phi(y,x_{n}).$$
 Since $\phi$ is lower semicontinuous respect to the first argument $y$, from Lemma \ref{lem8}, we deduce
 $$\partial A(y)=\lambda_{n}\partial_{2}f(x_{n},y)+\frac{1}{2}\nabla_{1}\phi(y,x_{n}).$$
 Let $u_{1}, u_{2}\in C$, $w_{1}\in\partial_{2}f(x_{n},u_{1})$ and $w_{2}\in \partial_{2}f(x_{n},u_{2})$, we have
 \begin{equation}\label{equa5}
 \langle w_{1},u_{1}-y\rangle\geq f(x_{n},u_{1})-f(x_{n},y), \:\: \: \forall y\in C,
 \end{equation}
 \begin{equation}\label{equa7}
 \langle -w_{2},y-u_{2}\rangle\geq f(x_{n},u_{2})-f(x_{n},y), \:\: \: \forall y\in C.
 \end{equation}
 Set $y=u_{2}$ in inequality (\ref{equa5}) and $y=u_{1}$ in inequality (\ref{equa7}), we get
 \begin{equation}\label{equa4}
\langle w_{1}-w_{2},u_{1}-u_{2}\rangle\geq 0.
\end{equation}
On the other hand, we have
\begin{equation}
\frac{1}{2}\nabla_{1}\phi(u_{i},x_{n})=Ju_{i}-Jx_{n},\:\:\: i=1,2.
\end{equation}
Therefore, using Lemma \ref{lem16}, there exists $\tau>0 $ such that
\begin{equation}\label{equa13}
\langle Ju_{1}-Ju_{2},u_{1}-u_{2}\rangle\geq \tau \|u_{1}-u_{2}\|^{2}.
\end{equation}
From inequalities (\ref{equa4}) and (\ref{equa13}), we obtain
\begin{equation}\label{equa26}
\langle (\lambda_{n} w_{1}+Ju_{1}-Jx_{n})-(\lambda_{n}w_{2}+Ju_{2}-Jx_{n}),u_{1}-u_{2}\rangle\geq\tau\|u_{1}-u_{2}\|^{2}.
\end{equation}
For all $u \in C$, put $T_{n}(u) := \lambda_{n}w + Ju - Jx_{n}$, where $w \in \partial_{2}f(x_{n},u)$. So $T_{n}(u) \subseteq \partial A(u)$ for all $u\in C$. Therefore it follows from inequality (\ref{equa26}) that
\begin{equation*}
\langle t^{n}(u_{1}) - t^{n}(u_{2}), u_{1} - u_{2} \rangle \geq\tau\| u_{1} - u_{2}\|,
\end{equation*}
 for all $u_{1}, u_{2} \in C$, all $t^{n}(u_{1}) \in T_{n}(u_{1})$ and $t^{n}(u_{2}) \in T_{n}(u_{2})$, this means $T_{n}$ is multivalued monotone.
\par
Using Lemma \ref{lem7} and Lemma \ref{lem8}, we have
$$y_{n}=\min_{y\in C}{\lambda_{n}f(x_{n},y)+\frac{1}{2}\phi(y,x_{n})},$$
if only if
$$0\in \partial A(y_{n})+N_{C}(y_{n}).$$
Which implies that $-T_{n}(y_{n})\subseteq N_{C}(y_{n})$, thus
\begin{equation}\label{equa27}
\langle t^{n}(y_{n}),y-y_{n}\rangle\geq 0,
\end{equation}
for all $y\in C$ and all $t^{n}(y_{n}) \in T_{n}(y_{n})$. Replacing $u_{1}$ and $u_{2}$ by $x_{n}$ and $y_{n}$ in inequality (\ref{equa26}), respectively  and interchanging $y$ by $x_{n}$  in inequality (\ref{equa27}), we have
\begin{equation}\label{equa28}
\langle t^{n}(x_{n}),x_{n}-y_{n}\rangle\geq\langle t^{n}(y_{n}),x_{n}-y_{n}\rangle+\tau\|x_{n}-y_{n}\|^{2}\geq\tau\|x_{n}-y_{n}\|^{2},
\end{equation}
for all $t^{n}(x_{n}) \in T_{n}(x_{n})$ and  all $t^{n}(y_{n}) \in T_{n}(y_{n})$. So, $t^{n}(x_{n})\in \partial A(x_{n})$, for all $t^{n}(x_{n}) \in T_{n}(x_{n})$ and since $\frac{1}{2}\nabla_{1}\phi(x_{n},x_{n})=0$, we obtain $t^{n}(x_{n})\in \lambda_{n}\partial_{2} f(x_{n},x_{n}).$ Since $x_{n}\rightarrow\bar{x}$, it follows from Lemma \ref{lem14} that
%for every $\varepsilon>0$, there exist $\eta>0$ and $n_{0}\in \mathbb{N}$ such that
$$\partial_{2}f(x_{n},x_{n})\subset\partial_{2}f(\bar{x},\bar{x}).$$
Thus, we can deduce from Lemma \ref{lem20} that $t^{n}(x_{n})$ is bounded, because of $\: 0<\lambda\leq\lambda_{n}\leq 1$. So, from inequality (\ref{equa28}), we obtain
$\tau\|x_{n}-y_{n}\|\leq\|t^{n}(x_{n})\|.$
Therefore,  we can conclude that $\{y_{n}\}$ is bounded. Since $\{z_{n}\}$ is a convex combination of $\{x_{n}\}$ and $\{y_{n}\}$, it is also bounded, hence, there exists a subsequence of $\{z_{n}\}$, again denoted by $\{z_{n}\}$, which converges weakly to $\bar{z}\in C$. In a similar way, it follows from lemmas \ref{lem20} and \ref{lem14} that the sequence $\{g_{n}\}$ is bounded.\\
If $x_{n}=y_{n}$ then we have $f(z_{n},x_{n})=0$ and $\sigma_{n}=0$ and if $x_{n}\neq y_{n}$, by the definition of $\sigma_{n}$, we obtain
$$\nu f(z_{n},x_{n})=\sigma_{n}\|g_{n}\|\|g_{n}\|\rightarrow0\:\:\:\Longrightarrow\:\:\: f(z_{n},x_{n})\rightarrow0,$$
since $\sigma_{n}\|g_{n}\|\rightarrow 0$ and $0<\nu<\frac{c^{2}}{2}$.
\par
Now, we show that $\bar{x}\in E(f)$ and $\|x_{n}-y_{n}\|\rightarrow 0$.
If $y_{n}=x_{n}$ then it follows from Lemma \ref{lem1} (i), that
\begin{equation}\label{equa44}
 \lambda_{n}f(x_{n},y)\geq0,
 \end{equation}
for all $y\in C.$ By letting $n\rightarrow\infty$ in inequality (\ref{equa44}) and using condition $(A3^{*})$, we get $\: f(\bar{x},y)\geq0$, because of $0<\lambda\leq\lambda_{n}\leq1$, i.e, $\bar{x}\in E(f).$
\par
Now, we let that $y_{n}\neq x_{n}$, since $f(z_{n},.)$ is convex, we obtain
\begin{equation*}
\begin{aligned}
\rho_{n}&f(z_{n},y_{n})+(1-\rho_{n})f(z_{n},x_{n})\\
&\geq f(z_{n},\rho_{n}y_{n}+(1-\rho_{n})x_{n})=f(z_{n},z_{n})=0.
\end{aligned}
\end{equation*}
Therefore,
\begin{equation}\label{equa45}
\rho_{n}[f(z_{n},x_{n})-f(z_{n},y_{n})]\leq f(z_{n},x_{n})\rightarrow0,
\end{equation}
as $n\rightarrow\infty$. By the Step $3.1$ of Algorithm $2$ and inequality (\ref{equa45}), we have
\begin{equation}\label{equa46}
\frac{\alpha\rho_{n}}{2\lambda_{n}}\phi(y_{n},x_{n})\leq\rho_{n}[f(z_{n},x_{n})-f(z_{n},y_{n})]\leq f(z_{n},x_{n})\rightarrow0.
\end{equation}
Now, we consider two cases:\\
{\bf Case $1$:} $\limsup_{n\rightarrow\infty}\rho_{n}>0$.\\
 In this case, there exists $\bar{\rho}>0$ and a subsequence of $\{\rho_{n}\}$, again denoted by $\{\rho_{n}\}$, such that $\rho_{n}\rightarrow\bar{\rho}$ and since $0<\lambda\leq\lambda_{n}\leq1$, from inequality (\ref{equa46}), we can conclude that
$$\phi(y_{n},x_{n})\rightarrow0.$$
Thus, from Lemma \ref{lem5}, we have $\|y_{n}-x_{n}\|\rightarrow0,$ which implies $y_{n}\rightarrow\bar{x}.$\\
{\bf Case $2$:} $\lim\limits_{n}\rho_{n}\rightarrow0$. \\
Let $m$ be the smallest nonnegative integer such that the Step 3.1 of Algorithm $2$ is satisfied, i.e.,
$$f(z_{n,m},x_{n})-f(z_{n,m},y_{n})\geq\frac{\alpha}{2\lambda}\phi(y_{n},x_{n}),$$
where $z_{n,m}=(1-\gamma^{m})x_{n}+\gamma^{m}y_{n}$. So,
\begin{equation}\label{equa47}
f(z_{n,m-1},x_{n})-f(z_{n,m-1},y_{n})<\frac{\alpha}{2\lambda_{n}}\phi(y_{n},x_{n}).
\end{equation}
On the other hand, setting $y=x_{n}$ in Lemma \ref{lem1} (i), condition $(A1)$ and equality (\ref{equa6}) imply that
$$-\lambda_{n}f(x_{n},y_{n})\geq\langle Jy_{n}-Jx_{n},y_{n}-x_{n}\rangle=\frac{1}{2}\phi(y_{n},x_{n})+\frac{1}{2}\phi(x_{n},y_{n}).$$
Therefore,
\begin{equation}\label{equa48}
\frac{1}{2}\phi(y_{n},x_{n})\leq-\lambda_{n}f(x_{n},y_{n}).
\end{equation}
From inequalities (\ref{equa47}) and (\ref{equa48}), we have
\begin{equation}\label{equa49}
f(z_{n,m-1},x_{n})-f(z_{n,m-1}, y_{n})<-\alpha f(x_{n},y_{n}).
\end{equation}
Taking the limits as $n\rightarrow\infty$ in above inequality,  we obtain $z_{n,m-1}\rightarrow\bar{x}$, since $\gamma^{m}=\rho_{n}\rightarrow0$. Because of $\{y_{n}\}$ is bounded,  there exists a subsequence of $\{y_{n}\}$, again denoted by $\{y_{n}\}$, which converges weakly to $\bar{y}\in C$. By letting $n \to \infty$ in inequality (\ref{equa49}) and using conditions $(A1)$ and $(A3^{*})$, we get
\begin{equation}
-f(\bar{x},\bar{y})\leq-\alpha f(\bar{x},\bar{y}).
\end{equation}
Which implies that $f(\bar{x},\bar{y})\geq0$, because of $\alpha\in (0,1)$. So, If we take the limits as $n\rightarrow\infty$ in inequality (\ref{equa48}), then we can conclude that
$\phi(y_{n},x_{n})\rightarrow0.$
Thus, from Lemma \ref{lem5}, we have
$\|y_{n}-x_{n}\|\rightarrow 0$,
which implies that $y_{n}\rightarrow\bar{x}$. By Lemma \ref{lem1}, we have
\begin{equation}\label{equa50}
\lambda_{n}[f(x_{n},y)-f(x_{n},y_{n})]\geq\langle Jy_{n}-Jx_{n},y_{n}-y\rangle,
\end{equation}
for all $y\in C$. By letting $n\rightarrow\infty$ in inequality (\ref{equa50}), it follows that $f(\bar{x},y)\geq0$, because of $0<\lambda\leq\lambda_{n}\leq1$, this means $\bar{x}\in E(f).$
\par
Now, we show that $\bar{x}\in F(S)$. Let $r_{1}=\sup_{n\geq0}\{\|x_{n}\|,\|w_{n}\|\}$ and $r_{2}=\sup_{n\geq0}\{\|w_{n}\|,\|Sw_{n}\|\}$. Using Lemma \ref{lem15}, there exists a continuous, strictly increasing and convex function\linebreak $g_{1}:[0,2r_{1}]\rightarrow\mathbb{R}$ with $g_{1}(0)=0$ such that for $x^{*}\in \Omega$, we get
\begin{equation}\label{eq.n1}
\begin{aligned}
\phi(x^{*},t_{n})&=\phi(x^{*},J^{-1}(\alpha_{n}Jx_{n}+(1-\alpha_{n})(\beta_{n}Jw_{n}+(1-\beta_{n})JSw_{n})))\\
&\leq\|x^{*}\|^{2}+\alpha_{n}\|x_{n}\|^{2}+(1-\alpha_{n})\beta_{n}\|w_{n}\|^{2}+(1-\alpha_{n})(1-\beta_{n})\|Sw_{n}\|^{2}\\
&\hspace{.5cm}-2\alpha_{n}\langle x^{*},Jx_{n}\rangle-2(1-\alpha_{n})\beta_{n}\langle x^{*},Jw_{n}
\rangle-2(1-\alpha_{n})(1-\beta_{n})\langle x^{*},JSw_{n}\rangle\\
&\hspace{.5cm}-\alpha_{n}(1-\alpha_{n})\beta_{n}g_{1}(\|Jw_{n}-Jx_{n}\|)\\
&\leq\phi(x^{*},x_{n})-\alpha_{n}(1-\alpha_{n})\beta_{n}g_{1}(\|Jw_{n}-Jx_{n}\|),
\end{aligned}
\end{equation}
and in a similar way, there exists a continuous, strictly increasing and convex function $g_{2}:[0,2r_{2}]\rightarrow\mathbb{R}$ with $g_{2}(0)=0$ such that for $x^{*}\in \Omega$, we obtain
\begin{equation}\label{eq.n2}
   \phi(x^{*},t_{n})\leq\phi(x^{*},x_{n})-(1-\alpha_{n})^{2}\beta_{n}(1-\beta_{n})g_{2}(\|Jw_{n}-JSw_{n}\|)=0.
\end{equation}
It follows from inequalities (\ref{eq.n1}) and (\ref{eq.n2}) that
\begin{equation}\label{equa30}
\alpha_{n}(1-\alpha_{n})\beta_{n}g_{1}(\|Jw_{n}-Jx_{n}\|)\leq\phi(x^{*},x_{n})-\phi(x^{*},t_{n}),
\end{equation}
\begin{equation}\label{equa31}
(1-\alpha_{n})^{2}\beta_{n}(1-\beta_{n})g_{2}(\|Jw_{n}-JSw_{n}\|)\leq\phi(x^{*},x_{n})-\phi(x^{*},t_{n}).
\end{equation}
Taking the limits as $n\rightarrow\infty$ in inequalities (\ref{equa30}) and (\ref{equa31}), using Lemma \ref{lem10} and equality (\ref{equa43}), we obtain
%\begin{equation*}
%\liminf_{n\rightarrow\infty}\alpha_{n}(1-\alpha_{n})\beta_{n}>0\quad \& \quad
%\liminf_{n\rightarrow\infty} (1-\alpha_{n})\beta_{n}(1-\beta_{n})>0,
%\end{equation*}
%it following from Lemma \ref{lem10} and inequality (\ref{equa43}) that
\begin{equation*}
\lim\limits_{n\rightarrow\infty}g_{1}(\|Jw_{n}-Jx_{n}\|)=0\quad \& \quad
\lim\limits_{n\rightarrow\infty}g_{2}(\|Jw_{n}-JSw_{n}\|)=0.
\end{equation*}
From the properties of $g_{1}$ and $g_{2}$, we have
\begin{equation*}
\lim\limits_{n\rightarrow\infty}\|Jw_{n}-Jx_{n}\|=0\quad \& \quad
\lim\limits_{n\rightarrow\infty}\|Jw_{n}-JSw_{n}\|=0.
\end{equation*}
Since $J^{-1}$ is uniformly norm-to-norm continuous on bounded sets, we obtain
\begin{equation}\label{equa32}
\lim\limits_{n\rightarrow\infty}\|w_{n}-x_{n}\|=0\quad \& \quad
\lim\limits_{n\rightarrow\infty}\|w_{n}-Sw_{n}\|=0.
\end{equation}
So, we  get $\bar{x}\in \hat{F}(S)$, because of $w_{n}\rightharpoonup\bar{x}$, therefore
using the definition of $S$, we have that $\bar{x}\in F(S).$
\par
  Setting $z=u^{*}$ in Lemma \ref{lem4}, since $x_{n+1}=R_{C_{n}\cap D_{n}}x_{0}$ and $\phi$ is continuous respect to the first argument, we obtain
$$\phi(u^{*},x_{0})\geq\lim\limits_{n\rightarrow\infty}\phi(x_{n+1},x_{0})=
\lim\limits_{n\rightarrow\infty}\phi(x_{n},x_{0})=\phi(\bar{x},x_{0}).$$
Also, since $u^{*}=R_{\Omega}x_{0}$, it follows from Lemma
\ref{eq38} that
$$\phi(u^{*},x_{0})\leq\phi(y,x_{0}),$$
for all $y\in \Omega$. Therefore $\bar{x}=u^{*}$ and consequently sequences
$\{x_{n}\}_{n=1}^{\infty}$, $\{y_{n}\}_{n=1}^{\infty}$, $\{z_{n}\}_{n=1}^{\infty}$, $\{w_{n}\}_{n=1}^{\infty}$ and $\{t_{n}\}_{n=1}^{\infty}$ converge strongly to $R_{\Omega}{x_{0}}$.
\end{proof}
\section{Numerical Example}
Now, we demonstrate theorems \ref{eq25} and \ref{thm1} with an example. Also, we compare the behavior of the sequence $\{x_{n}\}$ generated by algorithms $1$ and $2$.
\begin{exa}
Let $E=\mathbb{R}$ and $C=[-100,100]$. Define $f(x,y):=y^{2}-4x^{2}+3xy.$\\
We see that $f$ satisfies the conditions $(A1)-(A5)$ as follows:
 \begin{enumerate}
 \item[(A1)] $f(x,x):=x^{2}-4x^{2}+3x^{2}=0$ for all $x\in C$,
 \item[(A2)] If $f(x,y)=(y-x)(y+4x)\geq0$, then $$f(y,x)=(x-y)(x+4y)=(x-y)((y+4x)-3(x-y))=-f(x,y)-3(x-y)^{2}\leq0,$$
 for all $x, y\in C$, i.e., $f$ is pseudomonotone, while $f$ is not monotone.
 \item[(A3)]
 If $x_{n}\rightharpoonup\bar{x}$ and $y_{n}\rightharpoonup\bar{y}$, then $$f(x_{n},y_{n})=y_{n}^{2}-4x_{n}^{2}+3x_{n}y_{n}\rightarrow\bar{y}^{2}-4\bar{x}^{2}+3\bar{x}\bar{y}=f(\bar{x},\bar{y}),$$
 i.e., $f$ is jointly weakly continuous on $C\times C$.
\item[(A4)] Let $\theta\in (0,1)$. Since
\begin{equation*}
 \begin{aligned}
 f(x,\theta y_{1}+(1-\theta)y_{2})&=(\theta y_{1}+(1-\theta) y_{2})^{2}-4x^{2}+3x(\theta y_{1}+(1-\theta) y_{2})\\
 &\leq\theta(y_{1}^{2}-4x^{2}+3xy_{1})+(1-\theta)(y_{2}^{2}-4x^{2}+3xy_{2})\\
&=\theta f(x,y_{1})+(1-\theta)f(x,y_{2}),
 \end{aligned}
\end{equation*}
 so $f(x,.)$ is convex, also, $\liminf_{y\rightarrow\bar{y}}f(x,y)=f(x,\bar{y})$, hence $f(x,.)$ is lower semicontinuous. Since $\partial_{2}f(x,y)=2y+3x$, thus $f(x,.)$ is subdifferentiable on $C$ for each $x\in C.$
\item[(A5)] Since $\phi(y,x)=(y-x)^{2}$, we get
 \begin{equation}
\begin{aligned}
 f(x,y)+f(y,z)&=z^{2}-4x^{2}-3y^{2}+3xy+3zy\\
&= f(x,z)-\frac{3}{2}(y-x)^{2}-\frac{3}{2}(z-y)^{2}+\frac{3}{2}(x-z)^{2}\\
&\geq f(x,z)-\frac{3}{2}(y-x)^{2}-\frac{3}{2}(z-y)^{2},
 \end{aligned}
\end{equation}
i.e., $f$ satisfies the $\phi$-Lipschitz-type condition with $c_{1},c_{2}=\frac{3}{2}.$
 \end{enumerate}
Now, define $S:C\rightarrow C$ by $Sx=\frac{x}{5}$ for all $x\in C$, so $F(S)=\{0\}$ and
\begin{equation}
\begin{aligned}
\phi(0,Sx)&=\phi(0,\frac{x}{5})\\
&=0-2\langle 0,\frac{x}{5}\rangle+|\frac{x}{5}|^{2}\\
&\leq|x|^{2}\\
&=\phi(0,x),
\end{aligned}
\end{equation}
for all $x\in C$. Let $x_{n}\rightharpoonup p$ such that $\lim\limits_{n\rightarrow\infty}(Sx_{n}-x_{n})=0$,
this implies that $\hat{F}(S)=\{0\}$.\\
\begin{figure}[!h]
\includegraphics[totalheight=2.5in]{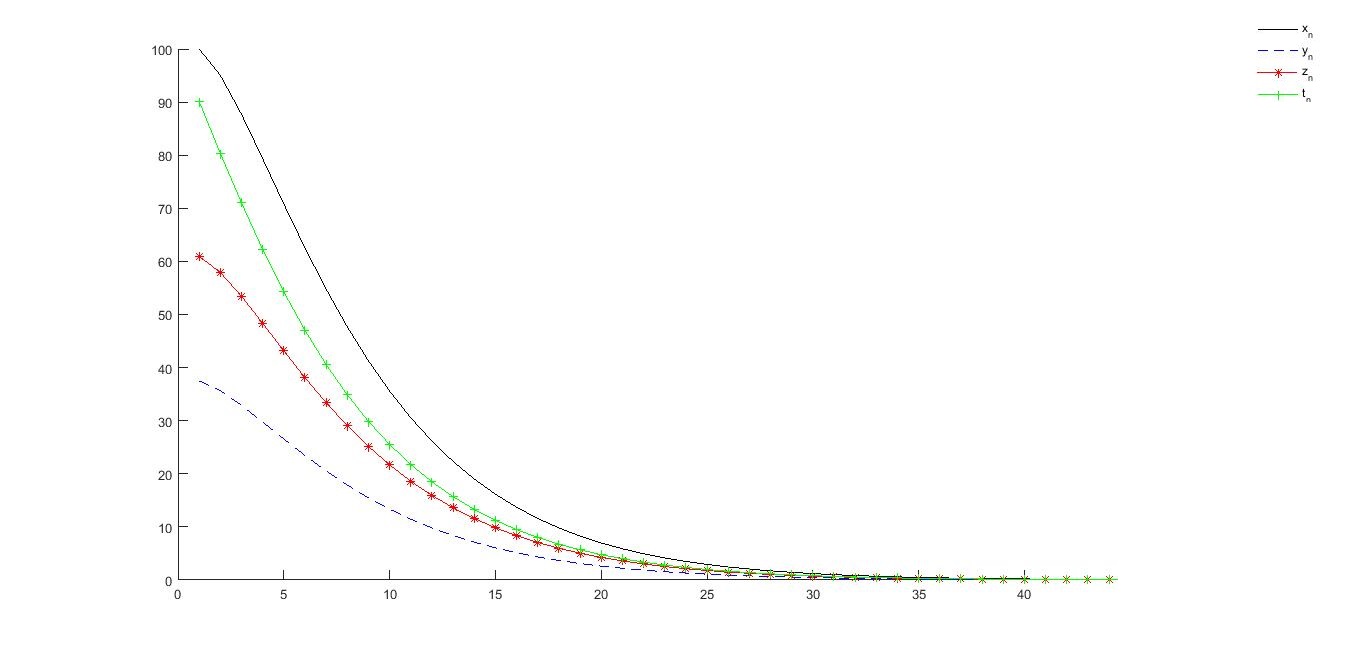}\vspace{-.75cm}\caption{Extragradient Algorithm}
\end{figure}
\begin{figure}[!h]
\includegraphics[totalheight=2.5in]{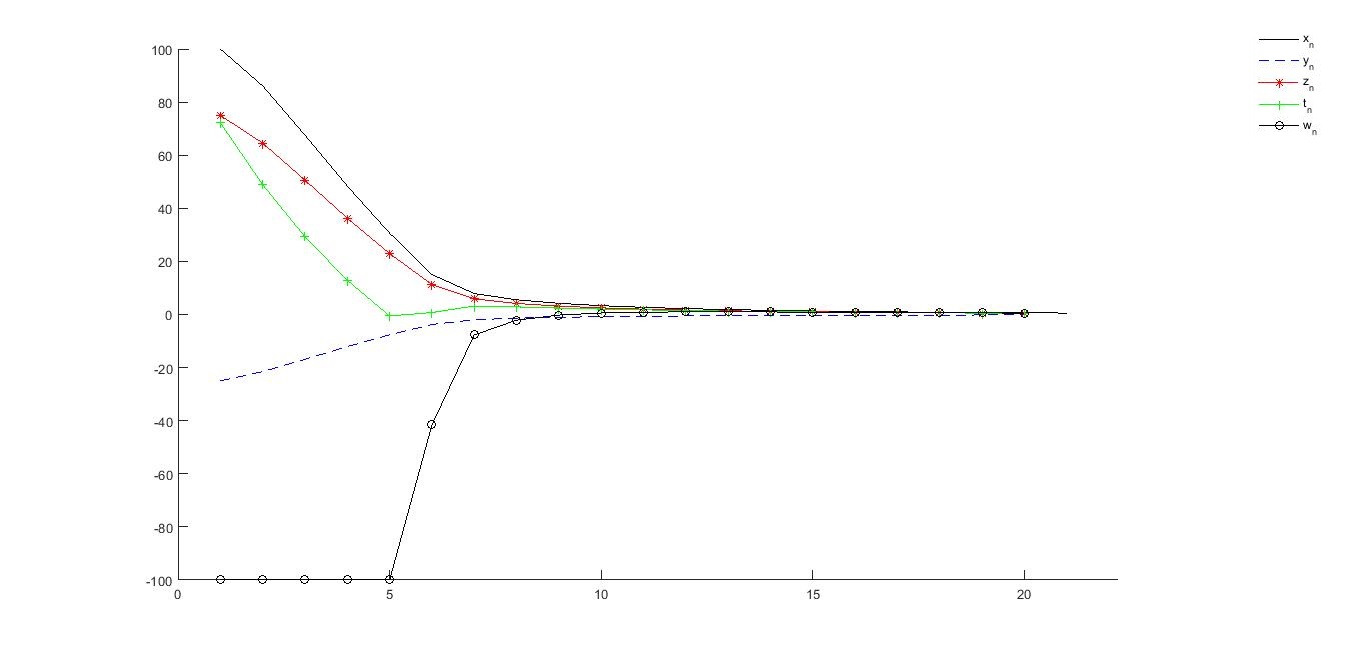}\vspace{-.75cm}\caption{Linesearch Algorithm}
\end{figure}
Thus, $\hat{F}(S)=F(S)$, i.e., $S$ is relatively nonexpansive
mapping. On the other hand, if for each $y\in C$, $f(x,y)\geq0$,
then $x=0$, i.e., $E(f)=\{0\}$ and consequently $\Omega=E(f)\cap
F(S)=\{0\}.$
%Now, in Extragradient Algorithm
Also, assume that $\alpha_{n}=\frac{1}{2}+\frac{1}{3+n}$,
$\beta_{n}=\frac{1}{3}+\frac{1}{4+n}$ for all $n\geq 0$ and
$x_{0}=100$.
\par
In Extragradient Algorithm, if $\lambda_{n}=\frac{1}{6}$, then we have
%\begin{equation}
%\begin{aligned}
 $$\frac{1}{6}f(x_{n},y_{n})+\frac{1}{2}(y_{n}-x_{n})^{2}=\min_{y\in C}\{\frac{1}{6}f(x_{n},y)+\frac{1}{2}(y-x_{n})^{2}\},$$
 i.e., $y_{n}=\frac{3}{8}x_{n},$ also $$ \frac{1}{6}f(y_{n},z_{n})+\frac{1}{2}(z_{n}-x_{n})^{2}=\min_{y\in C}\{\frac{1}{6}f(y_{n},y)+\frac{1}{2}(y-x_{n})^{2}\},$$
i.e., $z_{n}=\frac{39}{24}y_{n}=\frac{39}{64}x_{n},$ therefore
%\end{aligned}
%\end{equation}
\begin{equation*}
t_{n}=\alpha_{n}x_{n}+(1-\alpha_{n})[\beta_{n}z_{n}+\frac{1}{5}(1-\beta_{n})z_{n}],
\end{equation*}
and $x_{n+1}=R_{C_{n}\cap D_{n}}x_{0}$ or $|x_{n+1}-x_{0}|=\min_{z\in C_{n}\cap D_{n}}|z-x_{0}|$, where
\begin{equation*}
\begin{cases}
C_{n}=\{z\in C: |t_{n}-z|\leq |x_{n}-z|\},\\
D_{n}=\{z\in C: (x_{n}-z)(x_{0}-x_{n})\geq0\}.
\end{cases}
\end{equation*}
\begin{center}
\vspace{3mm}
\begin{tabular}{ c c c c c}
\hline Table1 & {\small Numerical Results for Algorithm $1$} &
\\
\cline{1-5}
n\hspace{0 cm}  & \hspace{-3 cm} {$x_{n}$} & \hspace{-4 cm} {$y_{n}$} &\hspace{0 cm} {$z_{n}$} &\hspace{1.5 cm} {$t_{n}$}\\
\hline\hline
0\hspace{0 cm} &\hspace{-3 cm}$100$&\hspace{-4 cm} $37.5$ &\hspace{0 cm}$60.94$ & \hspace{1.5 cm} $90.104$ \\
1\hspace{0 cm} &\hspace{-3 cm}$95.05$&\hspace{-4 cm}$35.65$ & \hspace{0 cm}$57.92$ & \hspace{1.5 cm}$80.36$\\
2\hspace{0 cm} &\hspace{-3 cm}$87.71$&\hspace{-4 cm}$32.89$ & \hspace{0 cm}$53.47$ & \hspace{1.5 cm}$71.016$\\
3\hspace{0 cm} &\hspace{-3 cm}$79.36$&\hspace{-4 cm}$29.61$ & \hspace{0 cm}$48.36$ &\hspace{1.5 cm}$62.277$\\
%4\hspace{0 cm} &\hspace{-3 cm}$72.6$&\hspace{0 cm}$-0.32$ & \hspace{0 cm}$0.2176$ & \hspace{0 cm}$-0.97e-3$\\
%5\hspace{0 cm} &\hspace{-3 cm}$64.6$&\hspace{0 cm}$-0.18$ & \hspace{0 cm}$0.1476$ & \hspace{0 cm}$-0.69e-4$\\
\hspace{0 cm} &\hspace{-3 cm}$\vdots$&\hspace{-4 cm}$\vdots$ & \hspace{0 cm}$\vdots$ & \hspace{1.5 cm}$\vdots$\\
37\hspace{0 cm} &\hspace{-3 cm}$0.2826$&\hspace{-4 cm} $0.1059$ &\hspace{0 cm}$0.1722$ & \hspace{1.5 cm} $0.1881$ \\
38\hspace{0 cm} &\hspace{-3 cm}$0.2353$&\hspace{-4 cm}$0.0882$ & \hspace{0 cm}$0.1434$ & \hspace{1.5 cm}$0.1565$\\
39\hspace{0 cm} &\hspace{-3 cm}$0.1959$&\hspace{-4 cm} $0.0734$ &\hspace{0 cm}$0.1194$ & \hspace{1.5 cm} $0.1302$ \\
40\hspace{0 cm} &\hspace{-3 cm}$0.163$&\hspace{-4 cm}$0.0611$ & \hspace{0 cm}$0.0993$ & \hspace{1.5 cm}$0.1082$\\
 \hspace{0 cm} &\hspace{-3 cm}$\vdots$&\hspace{-4 cm}$\vdots$ & \hspace{0 cm}$\vdots$ & \hspace{1.5 cm}$\vdots$\\
70\hspace{0 cm} &\hspace{-3 cm}$0.0003$ & \hspace{-4 cm}$0.000113$ & \hspace{0 cm}$0.000183$ &\hspace{1.5 cm}$0.000197$\\
%\hspace{0 cm} &\hspace{-3 cm}$0.2$&\hspace{-5 cm}$-0.05$ &  \hspace{0 cm}$0.1375$ &\hspace{2 cm}$0.1369$\\
71\hspace{0 cm} &\hspace{-3 cm}$0.0002 $&\hspace{-4 cm}$7.50e-05$ & \hspace{0 cm}$0.000122$ &\hspace{1.5 cm}$0.000131$\\
72\hspace{0 cm} &\hspace{-3 cm}$0.0001$&\hspace{-4 cm}$3.75e-05$ & \hspace{0 cm}$6.09e-05$ & \hspace{1.5 cm}$6.55e-05$\\
%\hspace{0 cm} &\hspace{-3 cm}$\vdots $&\hspace{-5 cm}$\vdots$ &\hspace{0 cm}$\vdots$ & \hspace{2 cm}$\vdots$\\
73\hspace{0 cm} &\hspace{-3 cm}$0$&\hspace{-4 cm}$0$ & \hspace{0 cm}$0$ & \hspace{1.5 cm}$0$\\
%99\hspace{-.3 cm} &\hspace{-1.7 cm}$2.669e-198 $&\hspace{-3.8 cm}$1.334e-198$ & \hspace{0 cm}$-1.927e-198$\\
%100\hspace{-.3 cm} &\hspace{-1.7 cm}$9.609e-201 $&\hspace{-3.8 cm}$4.804e-201$ & \hspace{0 cm}$-6.939e-201$\\
\hline
\end{tabular}
\end{center}
\begin{center}
\vspace{3mm}
\begin{tabular}{ c c c c c c}
\hline Table2 &  {\small  Numerical Results for Algorithm $2$}&
\\
\cline{1-6}
n\hspace{0 cm} & \hspace{-4cm} {$x_{n}$} & \hspace{-6 cm} {$y_{n}$} & \hspace{-1 cm} {$z_{n}$} & \hspace{1 cm} {$t_{n}$}  & \hspace{1.5 cm}{$w_{n}$}\\
\hline\hline
0\hspace{0 cm} &\hspace{-4 cm}$100$ & \hspace{-6 cm} $-25$ &\hspace{-1 cm} $75$ & \hspace{1 cm} $72.22$ & \hspace{1.5 cm} $-100$\\
1\hspace{0 cm} &\hspace{-4 cm}$86.11$&\hspace{-6 cm}$-21.53$ & \hspace{-1 cm}$64.53$ & \hspace{1 cm}$48.91$& \hspace{1.5 cm}$-100$\\
2\hspace{0 cm} &\hspace{-4 cm}$67.51$&\hspace{-6 cm}$-16.88$ & \hspace{-1 cm}$50.63$ & \hspace{1 cm}$29.26$& \hspace{1.5 cm}$-100$\\
3\hspace{0 cm} &\hspace{-4 cm}$48.38$&\hspace{-6 cm}$-12.097$ & \hspace{-1 cm}$36.29$ &\hspace{1 cm}$12.89$& \hspace{1.5 cm}$-100$\\
%4\hspace{0 cm} &\hspace{-3 cm}$72.6$&\hspace{0 cm}$-0.32$ & \hspace{0 cm}$0.2176$ & \hspace{0 cm}$-0.97e-3$\\
%5\hspace{0 cm} &\hspace{-3 cm}$64.6$&\hspace{0 cm}$-0.18$ & \hspace{0 cm}$0.1476$ & \hspace{0 cm}$-0.69e-4$\\
\hspace{0 cm} &\hspace{-4 cm}$\vdots$&\hspace{-6 cm}$\vdots$ & \hspace{-1 cm} $\vdots$ & \hspace{1 cm}$\vdots$& \hspace{1.5 cm}$\vdots$\\
37\hspace{0 cm} &\hspace{-4 cm}$0.0561$&\hspace{-6 cm} $-0.01403$ &\hspace{-1 cm} $0.0421$ & \hspace{1 cm} $0.0422$& \hspace{1.5 cm}$0.0553$ \\
38\hspace{0 cm} &\hspace{-4 cm}$0.0491$&\hspace{-6 cm}$-0.01228$ & \hspace{-1 cm} $0.0368$ & \hspace{1 cm}$0.0369$& \hspace{1.5 cm}$0.0485$\\
39\hspace{0 cm} &\hspace{-4 cm}$0.043$&\hspace{-6 cm} $-0.01075$ &\hspace{-1 cm} $0.0322$ & \hspace{1 cm} $0.0329$& \hspace{1.5 cm}$0.0426$ \\
40\hspace{0 cm} &\hspace{-4 cm}$0.0376$&\hspace{-6 cm}$-0.0094$ & \hspace{-1 cm} $0.0282$ & \hspace{1 cm}$0.0283$& \hspace{1.5 cm}$0.0373$\\
 \hspace{0 cm} &\hspace{-4 cm}$\vdots$&\hspace{-6 cm}$\vdots$ & \hspace{-1 cm}$\vdots$ & \hspace{1 cm}$\vdots$& \hspace{1.5 cm}$\vdots$\\
%35\hspace{0 cm} &\hspace{-4 cm}$0.3$ & \hspace{0 cm}$-0.04$ & \hspace{0 cm}$0.0484$ &\hspace{0 cm}$\vdots$\\
70\hspace{0 cm} &\hspace{-4 cm}$0.0003$&\hspace{-6 cm}$-7.50e-05$ &  \hspace{-1 cm}$0.00022$ &\hspace{1 cm}$0.00022$& \hspace{1.5 cm}$0.0003$\\
71\hspace{0 cm} &\hspace{-4 cm}$0.0002 $&\hspace{-6 cm}$-5.00e-05$ & \hspace{-1 cm}$0.00015$ &\hspace{1 cm}$0.00015$& \hspace{1.5 cm}$0.0002$\\
72\hspace{0 cm} &\hspace{-4 cm}$0.0001$&\hspace{-6 cm}$-2.50e-05$ & \hspace{-1 cm}$7.50e-5$ & \hspace{1 cm}$7.46e-5$ & \hspace{1.5 cm}$1.00e-04$\\
%\hspace{0 cm} &\hspace{-3 cm}$\vdots $&\hspace{-5 cm}$\vdots$ &\hspace{0 cm}$\vdots$ & \hspace{2 cm}$\vdots$\\
73\hspace{0 cm} & \hspace{-4 cm}$0$ &\hspace{-6 cm}$0$ & \hspace{-1 cm}$0$ & \hspace{1 cm}$0$ & \hspace{1.5 cm}$0$\\
%99\hspace{-.3 cm} &\hspace{-1.7 cm}$2.669e-198 $&\hspace{-3.8 cm}$1.334e-198$ & \hspace{0 cm}$-1.927e-198$\\
%100\hspace{-.3 cm} &\hspace{-1.7 cm}$9.609e-201 $&\hspace{-3.8 cm}$4.804e-201$ & \hspace{0 cm}$-6.939e-201$\\
\hline
\end{tabular}
\end{center}

 Since $\Omega=\{0\}$, we get $R_{\Omega}(x_{0})=0$.
  Moreover, numerical results for Algorithm $1$ show that the sequences $\{x_{n}\}$, $\{y_{n}\}$, $\{z_{n}\}$ and $\{t_{n}\}$ converge strongly to $0$.
\par
  In Linesearch Algorithm, $x_{n}$ is the same in Extragradient Algorithm.
  Assume that $\lambda_{n}=\frac{1}{2}$, $\alpha=\frac{1}{2}$, $\gamma=0.2$, $\nu=\frac{1}{4}$ and $c=1$. So,
   $$\frac{1}{2}f(x_{n},y_{n})+\frac{1}{2}(y_{n}-x_{n})^{2}=\min_{y\in C}\frac{1}{2}\{f(x_{n},y)+(y-x_{n})^{2}\},$$
 i.e., $y_{n}=-\frac{1}{4}x_{n},$ and $m$ is the smallest nonnegative integer such that
  $$(x_{n}-y_{n})(\frac{1}{2}x_{n}+\frac{3}{2}y_{n}+3z_{n})\geq0,$$
  where
 $$z_{n}=z_{n,m}=(1-(0.2)^{m})x_{n}+(0.2)^{m}y_{n}.$$
  Also, $g_{n}=2x_{n}+3z_{n}$ and $|w_{n}-(J^{-1}(Jx_{n}-\sigma_{n}g_{n}))|=\min_{z\in C}|z-(J^{-1}(Jx_{n}-\sigma_{n}g_{n})|$. \linebreak
  Since $y_{n}\neq x_{n}$, then $$\sigma_{n}=\frac{\frac{1}{4}(x_{n}^{2}-4z_{n}^{2}+3x_{n}z_{n})}{|g_{n}|},$$
  and
  \begin{equation*}
t_{n}=\alpha_{n}x_{n}+(1-\alpha_{n})[\beta_{n}w_{n}+\frac{1}{5}(1-\beta_{n})w_{n}].
\end{equation*}
Furthermore,  numerical results for Algorithm $2$ show that the
sequences $\{x_{n}\}$, $\{y_{n}\}$, $\{z_{n}\}$, $\{t_{n}\}$ and
$\{w_{n}\}$ converge strongly to $0$. By comparing Figure1 and
Figure2, we see that the speed of convergence of the sequence
$\{x_{n}\}$ generated by Linesearch Algorithm is equal to
Extragradient Algorithm. The computations associated with example
were performed using MATLAB (Step:$10^{-4}$) software.
   %Taking the limit as $k\rightarrow\infty$ in (\ref{eq71}), we obtain $\lim\limits_{k\rightarrow\infty}x^{k}=0$ and from (\ref{eq72}), we have $\lim\limits_{k\rightarrow\infty}y^{k}=\lim\limits_{k\rightarrow\infty}z^{k}=0$. See Figure1 and Figure2 for the values $x^{1}=2$ and $x^{1}=0.02$.
\end{exa}

{\footnotesize}

\end{document}